\documentclass[11pt]{article}
\usepackage{geometry}
\usepackage[utf8]{inputenc}
\usepackage{amssymb,amsmath,amsfonts, amsthm}
\usepackage{array}
\usepackage{mathrsfs}
\usepackage{graphicx,epsfig}
\usepackage{caption}
\usepackage{subcaption}
\usepackage{url}
\usepackage{lscape}
\usepackage{xspace}
\usepackage{algorithm}

\usepackage{algpseudocode}
\usepackage{hyperref}
\usepackage{mathtools}
\usepackage{multirow,multicol}
\usepackage{cleveref}
\usepackage{comment}
\usepackage{ulem}
\usepackage[usenames,dvipsnames]{color}
\usepackage{cancel}
\usepackage{mathdots}
\usepackage{diagbox}
\usepackage{extarrows}
\usepackage{arydshln}

\usepackage[title]{appendix}

\usepackage{blkarray}
\usepackage{makecell}
\usepackage{multicol}
\usepackage{amsopn}

\def\M{\mathcal{M}}

\def\a{\mathbf{a}}
\def\b{\mathbf{b}}
\def\c{\mathbf{c}}

\def\e{\mathbf{e}}
\def\f{\mathbf{f}}
\def\g{\mathbf{g}}

\def\k{\mathbf{k}}

\def\p{\mathbf{p}}
\def\q{\mathbf{q}}
\def\r{\mathbf{r}}
\def\s{\mathbf{s}}

\def\u{\mathbf{u}}
\def\v{\mathbf{v}}

\def\x{\mathbf{x}}
\def\y{\mathbf{y}}
\def\z{\mathbf{z}}
\def\U{\mathbf{U}}

\newtheorem{thm}{Theorem}[section]
\newtheorem{remark}[thm]{Remark}

\crefname{assumption}{Assumption}{Assumptions}

\crefname{lem}{Lemma}{Lemmas}
\crefname{thm}{Theorem}{Theorems}
\crefname{cor}{Corollary}{Corollaries}
\crefname{pr}{Proposition}{Propositions}
\crefname{remark}{Remark}{Remarks}
\crefname{algorithm}{Algorithm}{Algorithms}
\crefname{appendix}{}{Appendices}
\crefname{section}{Section}{Sections}
\crefname{figure}{Figure}{Figures}

\numberwithin{equation}{section}

\crefname{table}{Table}{Tables}

\title{Null-Space-Free 6D Spectral Embedding with Local Rayleigh Quotient Recovery for 3D Quasiperiodic Maxwell's Eigenproblems }

\author{
Teng-Chao Sun\footnote{School of Mathematics and Shing-Tung Yau Center, Southeast University, Nanjing 211189, People's Republic of China.}, \quad
Tiexiang Li\footnote{Corresponding author (txli@seu.edu.cn). School of Mathematics and Shing-Tung Yau Center, Southeast University, Nanjing 211189, People's Republic of China; Shanghai Institute for Mathematics and Interdisciplinary Sciences (SIMIS), Shanghai 200433, People's Republic of China.}, \quad
Wen-Wei Lin\footnote{Shanghai Institute for Mathematics and Interdisciplinary Sciences (SIMIS), Shanghai 200433, People's Republic of China; Department of Applied Mathematics, National Yang Ming Chiao Tung University, Hsinchu 300, Taiwan.} \quad
Xing-Long Lyu\footnote{School of Mathematical Sciences, Nanjing Normal University, Nanjing 210023, People's Republic of China}
}

\date{}

\begin{document}

\maketitle

\begin{abstract}
We develop a numerical framework for 3D quasiperiodic Maxwell eigenvalue problems based on a 6D periodic projection formulation. A projected Bloch-Fourier discretization yields a structured generalized eigenvalue problem with a large gradient kernel. We construct explicit orthonormal bases for the longitudinal and transverse subspaces to remove the gradient kernel exactly and reduce the original generalized eigenvalue problem to a null-space free standard eigenvalue problem. An explicit inverse representation of the reduced operator avoids nested linear solves and leads to an inverse Lanczos method with a Hermitian positive definite CG inner system whose condition number is bounded by the permittivity contrast. To recover the computed modes in physical space, the 6D Fourier eigenvectors are reconstructed on a 3D Yee grid by a separated multi-center Taylor expansion, which avoids the dense Fourier-to-grid phase matrix and remains practical when direct dense reconstruction becomes prohibitively expensive. The reconstructed fields are assessed using a separately discretized cropped Yee operator. Local weighted Rayleigh quotients provide spatially resolved spectral estimates, while their mass-weighted mean equals the cropped Yee quotient under a partition-of-unity condition. Numerical experiments demonstrate the accuracy and efficiency of the complete framework.
\end{abstract}


\section{Introduction}
\label{sec:introduction}

Quasiperiodic media possess long-range order without translational periodicity. Since the discovery of quasicrystals \cite{ShechtmanEtAl1984}, their geometric, spectral, and wave-propagation properties have attracted sustained interest; comprehensive reviews of photonic and phononic quasicrystals and their relation to periodic and disordered media can be found in \cite{SteurerSutterWidmer2007,VardenyNahataAgrawal2013,Edagawa2014}. Photonic quasicrystals exhibit complete or nearly isotropic band gaps \cite{ChanChanLiu1998,DyachenkoMiklyaevDmitrienko2007}, topologically protected transport \cite{BandresRechtsmanSegev2016}, and high-topological-charge lasing associated with noncrystallographic rotational symmetries \cite{ArjasEtAl2024}. These phenomena motivate reliable computation of quasiperiodic Maxwell spectra and modes.

A general quasiperiodic medium has no finite unit cell in physical space, and the standard Bloch formulation for periodic Maxwell systems \cite{JoannopoulosEtAl2008} therefore cannot be applied directly. Periodic approximants and large supercells provide one possible treatment, but the approximating domains may grow rapidly with the desired accuracy and may introduce severe band folding and mode mixing \cite{JiangLiZhang2025Approx,ZhangEtAl2022}. An alternative is to represent the quasiperiodic coefficient as the restriction of a periodic function defined in a higher-dimensional space. This viewpoint was used for the computation and visualization of photonic-quasicrystal spectra through a higher-dimensional Bloch formulation \cite{RodriguezEtAl2008}, and it also forms the basis of the projection method for general quasiperiodic systems \cite{JiangZhang2014}. Within this framework, subsequent work has developed numerical approximation and complexity analysis \cite{JiangLiZhang2024}, finite-point recovery of quasiperiodic functions \cite{JiangZhouZhang2024}, projection-based homogenization \cite{JiangEtAl2025Homogenization}, filtered projection methods for quasiperiodic eigenvalue problems \cite{JiangEtAl2025IWFPM}, and convergence analysis for quasiperiodic elliptic operators \cite{JiangTangZhaiZhou2026}. These developments provide a systematic numerical foundation for treating quasiperiodic problems directly in the higher-dimensional representation rather than through large periodic approximants.

For Maxwell eigenproblems, Gao, Xu, and Yang developed a divergence-free projection method for 3D quasiperiodic photonic crystals using a higher-dimensional periodic representation \cite{GaoXuYang2024}. Their method constructs divergence-free Fourier bases within the projected formulation and solves the resulting Maxwell eigenproblem in the divergence-free subspace. Related reduced projection and mode-selection strategies have also been developed for scalar quasiperiodic Schr\"odinger eigenproblems arising in photonic moir\'e lattices \cite{GaoXuYang2025Reduced}. 

These works demonstrate the effectiveness of projection-based formulations for quasiperiodic spectral problems. The present work instead focuses on the computational realization of the projected Maxwell eigenproblem and the interpretation of its eigenpairs in physical space. This leads to the following three computational issues.

\begin{itemize}
\item The first motivation is the efficient solution of the large scale eigenvalue problem (EVP) arising from the 6D projection formulation. The projected Bloch-Fourier discretization inherits the large gradient kernel of the Maxwell curl-curl operator, which complicates the computation of the smallest positive eigenvalues. Moreover, the 6D tensor-product Fourier discretization rapidly increases the dimension of the discrete eigenproblem, while inverse eigensolvers require repeated linear solves whose efficiency depends critically on the conditioning of the underlying systems. For periodic Maxwell problems, finite-difference and matrix analytic approaches have led to efficient null-space free eigensolvers based on explicit decompositions of the discrete curl-curl operator \cite{HuangEtAl2013,ChernEtAl2015,HuangEtAl2015,LyuEtAl2021FAME,LyuEtAl2022,TianEtAl2023,LyuLiLinLin2024}. These approaches provide useful guidance for the present high-dimensional quasiperiodic setting.

\item The second motivation is the recovery of physical electric fields from the 6D Fourier eigenvectors. A straightforward simultaneous evaluation on a 3D Yee grid forms a dense Fourier-to-grid phase matrix, with arithmetic cost and storage proportional to the product of the numbers of retained Fourier modes and physical sampling points. Blockwise direct evaluation reduces the auxiliary storage by avoiding the full phase matrix, but its arithmetic cost remains proportional to the product of the numbers of Fourier modes and physical sampling points.

\item The third motivation is the physical-space assessment of the computed 6D eigenpairs. A Fourier-truncated eigenpair obtained from the embedded problem does not by itself quantify how well the reconstructed field satisfies the 3D quasiperiodic Maxwell discretization. In our recent work \cite{SunLiLinLyu2026}, weighted Rayleigh quotients were introduced for reconstructed fields in 2D quasiperiodic Helmholtz problems, whereas the Maxwell setting requires an extension to coupled vector fields on staggered grids. A 3D spectral estimate and residual are therefore needed to assess the consistency between the 6D spectral computation and the reconstructed physical-space field.

\end{itemize}

Motivated by these considerations, we develop a unified computational framework for 3D quasiperiodic Maxwell EVPs based on a 6D Bloch-Fourier discretization, connecting positive-spectrum computation, physical-field reconstruction, and 3D consistency assessment. The main contributions are summarized as follows.

\begin{itemize}

\item The proposed modewise longitudinal-transverse decomposition removes the gradient kernel and yields a null-space free standard EVP. An explicit inverse representation eliminates nested solves, and the resulting inner linear system is Hermitian positive definite with a condition number bounded directly by the permittivity contrast. Large-scale experiments on eigenproblems of dimension exceeding $10^7$ show that the resulting inverse Lanczos solver achieves approximately a threefold speedup over the original shift-and-invert realization.

\item We develop a memory-efficient physical-field reconstruction based on a blockwise low-rank multi-center Taylor representation that exploits the local phase structure of the projected Fourier expansion. The resulting formulation avoids the full dense Fourier-to-grid phase matrix and enables efficient recovery of Fourier eigenvectors on large 3D staggered Yee grids.

\item A 3D local weighted Rayleigh quotient (LWRQ) framework is developed for the reconstructed staggered Maxwell fields, providing a physical-space spectral estimate together with a consistency assessment of the reconstructed modes against the 3D Yee discretization.
\end{itemize}

The remainder of this paper is organized as follows. Section~\ref{sec:Maxwell_problem} presents the 6D projection formulation and Bloch-Fourier discretization. Sections~\ref{sec:null_space_free} and~\ref{sec:inverse_lanczos} develop the null-space free reduction and inverse Lanczos solver, respectively. Section~\ref{sec:LWRQ_error_analysis} presents the physical-space reconstruction and LWRQ analysis. Section~\ref{sec:numerical_experiments} reports the numerical results, and Section~\ref{sec:conclusion} concludes the paper.

\noindent\textit{Notation.}
Bold lowercase letters denote vectors. The symbols $(\cdot)^{\mathsf T}$ and $(\cdot)^*$ denote the transpose and conjugate transpose, respectively. The notation $\|\cdot\|_2$ denotes the Euclidean vector norm or the induced matrix norm, as appropriate. We use $\imath=\sqrt{-1}$ for the imaginary unit.

\section{6D Projection Formulation and Bloch-Fourier Discretization of the 3D Quasiperiodic Maxwell's Problem}
\label{sec:Maxwell_problem}
We consider the 3D quasiperiodic Maxwell's eigenvalue problem
\begin{equation}
\label{eq:3d_maxwell_qp}
\begin{cases}
\nabla_{\r} \times \left( \mu^{-1}(\r)\, \nabla_{\r}\times \u(\r) \right) =
\omega^2 \varepsilon(\r)\u(\r),
\qquad \r = (r_1, r_2, r_3)^{\mathsf{T}}\in\mathbb R^3, \\[2mm]
\nabla_{\r}\cdot \left( \varepsilon(\r)\u(\r) \right) =0 .
\end{cases}
\end{equation}
Here $\u=(u_1,u_2,u_3)^{\mathsf T}$ denotes the electric field. We consider nonmagnetic media with $\mu(\r)=I_3$ and a uniformly symmetric positive definite quasiperiodic permittivity tensor $\varepsilon(\r)\in\mathbb R^{3\times3}$.We write $\lambda=\omega^2$ for the spectral parameter and $\widetilde\lambda$ for an eigenvalue of the fixed Fourier-pseudospectral discretization studied below.

We use a higher-dimensional projection framework \cite{RodriguezEtAl2008,JiangZhang2014,JiangLiZhang2024} to lift the 3D quasiperiodic Maxwell problem to a 6D periodic problem, where the lifted field is represented in Bloch form. We retain the full three-component tensor-product Fourier space and exploit the modewise longitudinal-transverse structure of the resulting curl-curl matrix. This formulation leads to the discrete compatibility and mass-weighted divergence relations used in the null-space free reduction of Section~\ref{sec:null_space_free}.

Let $\x = (x_1,\ldots,x_6)^{\mathsf T}\in\mathbb R^6$. We introduce the projection matrix
\begin{equation}
\label{eq:P_definition}
P=
[\p_1^{\mathsf T},\ldots,\p_6^{\mathsf T}]^{\mathsf T}
=
[\a_1,\a_2,\a_3]
\in\mathbb R^{6\times3},
\end{equation}
where $\p_\ell^{\mathsf T}$, $\ell=1,\ldots,6$, are the row vectors of $P$, and $\a_1,\a_2,\a_3\in\mathbb R^6$ are its column vectors. We assume that $\operatorname{rank}(P)=3$.

The physical and higher-dimensional coordinates are related by
\begin{equation}
\label{eq:projection_relation}
\x=P\r.
\end{equation}
The quasiperiodic permittivity is represented as the restriction of a periodic tensor field $\mathcal E(\x)\in\mathbb R^{3\times3}$ on the 6D torus $\mathcal T^6$ with period vector $\mathbf T=(T_1,\ldots,T_6)$:
\begin{equation}
\label{eq:epsilon_lifting}
\varepsilon(\r)=\mathcal E(P\r),
\qquad
\mathcal E(\x+T_\ell\e_\ell)=\mathcal E(\x),
\qquad
\ell=1,\ldots,6.
\end{equation}
We assume that $\mathcal E(\x)$ is uniformly symmetric positive definite and that the physical frequency generators
$\{(2\pi/T_\ell)\p_\ell\}_{\ell=1}^6$
are rationally independent, i.e.,
$$
\sum_{\ell=1}^6
m_\ell\frac{2\pi}{T_\ell}\p_\ell=0,
\qquad
m_\ell\in\mathbb Z,
\quad\Longrightarrow\quad
m_1=\cdots=m_6=0.
$$

Let $\U:\mathbb R^6\to\mathbb C^3$ be a lifted vector field, and define its restriction to the physical quasiperiodic cut by $ \u(\r)=\U(P\r)$. By the chain rule, differentiation along the physical cut gives
\begin{equation}
\label{eq:projected_gradient_operator}
\nabla_{\r} =
P^{\mathsf T}\nabla_{\x} =
\begin{pmatrix}
\a_1^{\mathsf T}\nabla_{\x}\\
\a_2^{\mathsf T}\nabla_{\x}\\
\a_3^{\mathsf T}\nabla_{\x}
\end{pmatrix}
=: \nabla_P.
\end{equation}

The corresponding projected curl-curl operator satisfies
\begin{equation}
\label{eq:projected_curl_curl}
\nabla_{\r}\times\nabla_{\r}\times\u(\r)
=
\left[
\nabla_P\times
\left(
\nabla_P\times\U
\right)
\right](P\r).
\end{equation}
Therefore, the corresponding 6D periodic Maxwell problem is
\begin{equation}
\label{eq:projected_maxwell_evp}
\nabla_P\times
\left(
\nabla_P\times\U(\x)
\right)
=
\lambda\mathcal E(\x)\U(\x),
\qquad
\x\in\mathbb R^6,
\end{equation}
with the Bloch quasiperiodicity specified below.

We impose a Bloch representation on the lifted problem and discretize its periodic factor by a tensor-product Fourier-pseudospectral method on $\mathcal T^6$. General Fourier spectral methods are described in \cite{CanutoEtAl2006}, while their use in higher-dimensional quasiperiodic projection methods and the associated convergence analysis are discussed in \cite{JiangLiZhang2024,JiangTangZhaiZhou2026}. For sufficiently smooth periodic lifts, Fourier expansions provide high-order function approximation, with exponential rates under analyticity assumptions. Here we study the fixed discrete problem and its physical-space consistency, without asserting convergence of ordered eigenvalues as the Fourier cutoff increases. Following the standard and higher-dimensional Bloch formulations \cite{JoannopoulosEtAl2008,RodriguezEtAl2008,GaoXuYang2024}, we write the lifted field as
\begin{equation}
\label{eq:lifted_bloch_form}
\U(\x) = e^{\imath\k\cdot\x}\U_p(\x),
\qquad
\U_p(\x+T_\ell\e_\ell)=\U_p(\x),
\qquad
\ell=1,\ldots,6,
\end{equation}
where $\U_p:\mathcal T^6\to\mathbb C^3$ denotes the periodic factor of the lifted Bloch field, $\mathbb B^6=\prod_{\ell=1}^{6}[-\pi/T_\ell,\pi/T_\ell)$ is the first Brillouin zone, $\k\in\mathbb B^6$ is the lifted Bloch vector, and $\q=P^{\mathsf T}\k\in\mathbb R^3$ is the corresponding physical Bloch wave vector. Accordingly, the projected Bloch differential operator can be written as
\begin{equation}
\label{eq:projected_bloch_gradient}
\nabla_{P,\k} :=
P^{\mathsf{T}}(\nabla_{\x}+\imath\k) =
\nabla_P+\imath\q.
\end{equation}
Hence $\U$ is $\k$-quasiperiodic with respect to the 6D lattice periods, while its restriction to the physical cut carries the physical Bloch wave vector $\q$.

Expanding the periodic factor $\U_p$ in the tensor-product Fourier basis gives
\begin{equation}
\label{eq:bloch_fourier_expansion_U}
\U(\x) =
\sum_{\boldsymbol{\xi}\in\mathcal J_v}
\U_{\boldsymbol{\xi}} e^{\imath(\k+\boldsymbol{\xi})\cdot\x},
\end{equation}
where
\begin{equation}
\label{eq:reciprocal_index_set_JN}
\mathcal J_v =
\left\{
\left( \frac{2\pi}{T_1}j_1,\ldots, \frac{2\pi}{T_6}j_6 \right)^{\mathsf T}
: j_\ell=-N,\ldots,N-1, \; \ell = 1, \ldots, 6
\right\},
\qquad
N_F=(2N)^6.
\end{equation}

For each retained Fourier mode, define
\begin{equation}
\label{eq:effective_wave_vector_and_curl_symbol}
\q_{\boldsymbol{\xi}} =
P^{\mathsf T}(\k+\boldsymbol{\xi}) =
(q_{\boldsymbol{\xi},1},q_{\boldsymbol{\xi},2},q_{\boldsymbol{\xi},3})^{\mathsf T},
\qquad
Q_{\boldsymbol{\xi}} =
\begin{pmatrix}
0 & -q_{\boldsymbol{\xi},3} & q_{\boldsymbol{\xi},2}\\
q_{\boldsymbol{\xi},3} & 0 & -q_{\boldsymbol{\xi},1}\\
-q_{\boldsymbol{\xi},2} & q_{\boldsymbol{\xi},1} & 0
\end{pmatrix}.
\end{equation}
Hence the projected curl-curl operator acts modewise through $Q_{\boldsymbol{\xi}}^{\mathsf T}Q_{\boldsymbol{\xi}}$.

We order the Fourier coefficients by field component as $\U=(\U_1^{\mathsf T},\U_2^{\mathsf T},\U_3^{\mathsf T})^{\mathsf T}$, where $\U_j\in\mathbb C^{N_F}$ for $j=1,2,3$. Let $Q_j=\operatorname{diag}(q_{\boldsymbol{\xi},j})_{\boldsymbol{\xi}\in\mathcal J_v}$ for $j=1,2,3$. The  resulting stiffness matrix is
\begin{equation}
\label{eq:K_Q_block_form}
K
=
\begin{pmatrix}
Q_2^2+Q_3^2 & -Q_1Q_2 & -Q_1Q_3\\
-Q_1Q_2 & Q_1^2+Q_3^2 & -Q_2Q_3\\
-Q_1Q_3 & -Q_2Q_3 & Q_1^2+Q_2^2
\end{pmatrix}.
\end{equation}

The permittivity term is discretized pseudospectrally on the same $(2N)^6$ tensor grid. Let $\mathscr F_N$ denote the corresponding 6D discrete Fourier transform, and let
$\x_{\boldsymbol m}=(m_1T_1/(2N),\ldots,m_6T_6/(2N))^{\mathsf T}$,
with $\boldsymbol m=(m_1,\ldots,m_6)^{\mathsf T}\in\{0,\ldots,2N-1\}^6$, denote the tensor-grid points. For $i,j=1,2,3$, let
$\mathcal E_N^{(i,j)}
=
\left(
\mathcal E^{(i,j)}(\x_{\boldsymbol m})
\right)_{\boldsymbol m}$
denote the sampled values of the $(i,j)$ entry of the permittivity tensor, and define
\begin{equation}
\label{eq:pseudospectral_mass_blocks}
D_{\mathcal E}^{(i,j)}
=
\operatorname{diag}\left(\mathcal E_N^{(i,j)}\right),
\qquad
M^{(i,j)}
=
\mathscr F_N
D_{\mathcal E}^{(i,j)}
\mathscr F_N^{-1}.
\end{equation}
Equivalently, the action of each block is
\begin{equation}
\label{eq:pseudospectral_mass_action}
M^{(i,j)}\v
=
\mathscr F_N
\left[
\mathcal E_N^{(i,j)}
\odot
\mathscr F_N^{-1}\v
\right].
\end{equation}
Under the component-wise ordering, the discrete mass matrix is therefore $M = \left[ M^{(i,j)} \right]_{i,j=1}^{3}$, and the Fourier-pseudospectral discretization is
\begin{equation}
\label{eq:discrete_gevp}
K\U = \widetilde{\lambda}M\U.
\end{equation}

The action of $M$ is evaluated by component-wise inverse FFTs, pointwise multiplication by the sampled $3\times3$ permittivity tensors, and component-wise FFTs. Each application of $M$ costs $\mathcal O(N_F\log N_F)$ operations and requires $\mathcal O(N_F)$ auxiliary storage. Since $\mathcal E(\x_{\boldsymbol m})$ is uniformly positive definite and the discrete Fourier transform is unitary up to normalization, $M$ is Hermitian positive definite, while $K$ is Hermitian positive semidefinite.

Assume that $\q_{\boldsymbol{\xi}}\neq0$ for all $\boldsymbol{\xi}\in\mathcal J_v$. For each retained mode, $Q_{\boldsymbol{\xi}}^{\mathsf T}Q_{\boldsymbol{\xi}}$ has eigenvalues $0$, $\|\q_{\boldsymbol{\xi}}\|_2^2$, and $\|\q_{\boldsymbol{\xi}}\|_2^2$. Hence, each retained mode contributes one longitudinal zero mode and two transverse positive modes, so that
$$
\dim\ker(K)=N_F,
\qquad
\operatorname{rank}(K)=2N_F.
$$

The longitudinal null space is spanned by the columns of
\begin{equation}
\label{eq:null_space_U0}
G =
\begin{pmatrix}
Q_1\\ Q_2\\ Q_3
\end{pmatrix}
=
\begin{pmatrix}
\operatorname{diag}\left(q_{\boldsymbol{\xi},1}\right)_{\boldsymbol{\xi}\in\mathcal J_v}\\
\operatorname{diag}\left(q_{\boldsymbol{\xi},2}\right)_{\boldsymbol{\xi}\in\mathcal J_v}\\
\operatorname{diag}\left(q_{\boldsymbol{\xi},3}\right)_{\boldsymbol{\xi}\in\mathcal J_v}
\end{pmatrix}
\in\mathbb C^{3N_F\times N_F}.
\end{equation}

\paragraph{Discrete compatibility and divergence constraint}
With $G$ in \eqref{eq:null_space_U0} serving as the discrete gradient operator, define the global discrete curl and divergence operators by
\begin{equation}
\label{eq:discrete_operator}
Q =
\begin{pmatrix}
0 & -Q_3 & Q_2 \\
Q_3 & 0 & -Q_1 \\
-Q_2 & Q_1 & 0
\end{pmatrix},
\qquad
D=G^{*}.
\end{equation}
Here $Q$ acts mode by mode through the Fourier curl symbols $Q_{\boldsymbol{\xi}}$, while $G$ collects the projected wave-vector components $\q_{\boldsymbol{\xi}}=P^{\mathsf{T}}(\k+\boldsymbol{\xi})$. With this notation, the stiffness matrix satisfies
\begin{equation} \label{eq:double_curl_operator}
    K=Q^{*}Q.
\end{equation}

\begin{thm}
\label{thm:discrete_divergence_free}
The discrete operators satisfy
\begin{equation}
\label{eq:discrete_operator_identity_1}
QG=0,
\qquad
DQ^{*}=0.
\end{equation}
Moreover, if $\U$ is an eigenvector of \eqref{eq:discrete_gevp} associated with a positive eigenvalue $\widetilde{\lambda}>0$, then
\begin{equation}
\label{eq:discrete_operator_identity_2}
DM\U
=
G^{*}M\U
=
0.
\end{equation}
\end{thm}

\begin{proof}
For each $\boldsymbol{\xi}\in\mathcal J_v$, $Q_{\boldsymbol{\xi}}\q_{\boldsymbol{\xi}}=0$. After assembling all modes in the component-wise ordering, this gives $QG=0$. Taking the Hermitian transpose yields $G^{*}Q^{*}=0$, and hence $DQ^{*}=0$ because $D=G^{*}$.

Since $KG=0$ and $K$ is Hermitian, $G^{*}K=0$. For a positive eigenpair $K\U=\widetilde{\lambda}M\U$ with $\widetilde{\lambda}>0$, left multiplication by $G^{*}$ gives
$$
0=G^{*}K\U = \widetilde{\lambda}G^{*}M\U.
$$
Thus $G^{*}M\U=DM\U=0$, which is the discrete counterpart of the physical divergence constraint in \eqref{eq:3d_maxwell_qp}.
\end{proof}

Thus, the fixed Bloch-Fourier discretization retains the gradient kernel of the Maxwell curl-curl operator, while every positive eigenvector satisfies the corresponding mass-weighted discrete divergence constraint. The next section exploits this longitudinal-transverse structure to remove the gradient kernel and derive a null-space free formulation for the positive spectrum.

\section{null-space Free Reduction to a Standard Eigenvalue Problem}
\label{sec:null_space_free}

The generalized eigenvalue problem (GEVP) $K \U=\widetilde\lambda M \U$ inherits the gradient kernel of the discrete curl-curl operator. We follow the range-space strategy underlying null-space free Maxwell solvers \cite{HuangEtAl2013,ChernEtAl2015,LyuEtAl2021FAME} and give its explicit realization for the projected Fourier symbols used here. The construction identifies orthonormal longitudinal and transverse bases and yields a standard EVP preserving the complete positive spectrum of the fixed GEVP.

Define
\begin{equation}
\label{eq:Qs_Qp_def}
Q_q = Q_1^{*}Q_1+Q_2^{*}Q_2+Q_3^{*}Q_3,
\qquad
Q_s=Q_1+Q_2+Q_3,
\qquad
Q_p=Q_sQ_s^{*}.
\end{equation}

\begin{remark} 
\label{remark:exceptional_effective_wave_vectors} 
The formulas below use the reference direction $\mathbf t=(1,1,1)^{\mathsf T}$, assumed nonparallel to every retained effective wave vector. A different direction gives an analogous construction with corresponding changes to the reference block and normalization factors. We also assume $\q_{\boldsymbol\xi}\neq0$ for all retained modes. A zero effective wave vector, as may occur at the $\Gamma$ point, lies outside this setting and its three-component mode must instead be assigned to the null space. 
\end{remark}

Under these assumptions, the normalization factors below are strictly positive, so their inverse square roots are well defined.

We normalize the longitudinal basis in \eqref{eq:null_space_U0} to obtain the orthonormal null-space basis
\begin{equation}
\label{eq:null_space_orthonormal_basis}
\U_0
=
GQ_q^{-\frac12}
=
\begin{pmatrix}
Q_1\\
Q_2\\
Q_3
\end{pmatrix}
Q_q^{-\frac12}.
\end{equation}

We construct an orthonormal basis $\U_r=[\U_1,\U_2]\in\mathbb C^{3N_F\times2N_F}$ for $\operatorname{range}(K)$. Using the reference block $\mathbf T_1=[I,I,I]^{\mathsf T}\in\mathbb R^{3N_F\times N_F}$, define
\begin{subequations}
\label{total:transverse_basis_construction}
\begin{align}
\widehat{\U}_1
&=
(I-\U_0\U_0^{*})\mathbf T_1Q_q
=
\begin{pmatrix}
Q_q-Q_1Q_s^{*}\\
Q_q-Q_2Q_s^{*}\\
Q_q-Q_3Q_s^{*}
\end{pmatrix},
\label{eq:U1_projection}
\\
\U_1
&=
\widehat{\U}_1(3Q_q^2-Q_qQ_p)^{-\frac12},
\label{eq:U1_basis}
\\
\widehat{\U}_2
&=
Q^{*}\mathbf T_1
=
\begin{pmatrix}
Q_3^{*}-Q_2^{*}\\
Q_1^{*}-Q_3^{*}\\
Q_2^{*}-Q_1^{*}
\end{pmatrix},
\label{eq:U2_hat_def}
\\
\U_2
&=
\widehat{\U}_2(3Q_q-Q_p)^{-\frac12}.
\label{eq:U2_basis}
\end{align}
\end{subequations}
Here $Q$ is defined in \eqref{eq:discrete_operator}; the normalization factors in \eqref{eq:U1_basis} and \eqref{eq:U2_basis} make both transverse blocks orthonormal.

\begin{thm}
\label{thm:null_space_free_basis}
With the reference direction chosen as in Remark~\ref{remark:exceptional_effective_wave_vectors}, let $\U_0$, $\U_1$, and $\U_2$ be defined by \eqref{eq:null_space_orthonormal_basis}, \eqref{eq:U1_basis}, and \eqref{eq:U2_basis}, respectively, and set $\U_r=[\U_1,\U_2]$. Then $[\U_0,\U_1,\U_2]$ is unitary. Moreover, $ \operatorname{range}(\U_r)=\operatorname{range}(K)$,  and
\begin{equation}
\label{eq:K_range_decomp}
K=\U_rQ_r\U_r^*,
\qquad
Q_r=\operatorname{diag}(Q_q,Q_q)>0.
\end{equation}
\end{thm}

\begin{proof}
By construction, $K\U_0=0$ and $\U_0^*\U_0=I$, so $\U_0$ spans $\ker(K)$. Moreover, direct calculation gives
$$
\widehat{\U}_1^*\widehat{\U}_1
=
3Q_q^2-Q_qQ_p,
\qquad
\widehat{\U}_2^*\widehat{\U}_2
=
3Q_q-Q_p,
$$
which implies that $\U_1$ and $\U_2$ are orthonormal.
The constructions of $\U_1$ and $\U_2$ are orthogonal to $\U_0$ and to
each other; hence $[\U_0,\U_1,\U_2]$ is unitary.
Finally, using the block form of $K$,
$$
K\U_1=\U_1Q_q,\qquad K\U_2=\U_2Q_q,
$$
which yields
$K=\U_rQ_r\U_r^*$ with
$Q_r=\operatorname{diag}(Q_q,Q_q)$.
\end{proof}

For every positive eigenvector of \eqref{eq:discrete_gevp}, Theorem~\ref{thm:discrete_divergence_free} gives $G^{*}M\U=0$. Since $\operatorname{range}(\U_r)=\ker(G^{*})$, it follows that $M\U\in\operatorname{range}(\U_r)$. Hence every positive eigenvector can be written as
\begin{equation}
\label{eq:null_space_free_transform}
\U = M^{-1}\U_rQ_r^{\frac12}\v .
\end{equation}

Substituting \eqref{eq:null_space_free_transform} into the GEVP and premultiplying by $Q_r^{-1/2}\U_r^*$ gives
\begin{equation}
\label{eq:null_space_free_EVP}
K_r\v = \widetilde{\lambda}\v,
\qquad
K_r = Q_r^{\frac12}\U_r^{*}M^{-1}\U_rQ_r^{\frac12}.
\end{equation}

Thus, \eqref{eq:null_space_free_EVP} gives an exact null-space free reduction that removes the gradient kernel while preserving the complete positive spectrum of the fixed Fourier-pseudospectral discretization. The next section develops an inverse Lanczos method for its smallest eigenvalues.

\section{Inverse Lanczos Method with Orthogonalized Inner Solves}
\label{sec:inverse_lanczos}

Since $M$ and $Q_r$ are Hermitian positive definite and $\U_r$ has orthonormal columns, the reduced matrix $K_r$ in \eqref{eq:null_space_free_EVP} is Hermitian positive definite. We compute the smallest eigenvalues of $K_r$ by applying Lanczos to $K_r^{-1}$, whose largest eigenvalues are the reciprocals of the smallest eigenvalues of $K_r$. Although inverse Lanczos combined with null-space free Maxwell decompositions has been effective for periodic problems \cite{HuangEtAl2013,LyuEtAl2021FAME}, a naive application of $K_r^{-1}$ here would lead to nested linear solves, because each application of $K_r$ already involves $M^{-1}$. We therefore derive an explicit inverse representation whose dominant iterative solve is Hermitian positive definite.

Set $V=\U_rQ_r^{1/2}$. If necessary, apply a modewise row permutation to $V$ together with the corresponding symmetric permutation of $M$, and partition
\begin{subequations}
\label{eq:V_M_block_def}
\begin{equation}
\label{eq:V_block_def}
V =
\begin{pmatrix}
V_1\\
V_2
\end{pmatrix}
=
\U_rQ_r^{\frac12}, \; \text{with} \;
V_1\in\mathbb C^{2N_F\times 2N_F},
\qquad
V_2\in\mathbb C^{N_F\times2N_F},
\end{equation}
\begin{equation}
\label{eq:M_block_def_lanczos}
M
=
\begin{pmatrix}
M_1 & M_3\\
M_3^{*} & M_2
\end{pmatrix}
>0,
\qquad
M_1\in\mathbb C^{2N_F\times2N_F},
\qquad
M_2\in\mathbb C^{N_F\times N_F}.
\end{equation}
\end{subequations}
The permutation is chosen so that $V_1$ is nonsingular. Such a choice is possible because the two transverse vectors are linearly independent for every retained mode. With this notation, $K_r=V^{*}M^{-1}V$, and we define $W=V_1^{-*}V_2^{*}$. Because $Q_1$, $Q_2$, $Q_3$, and the normalization factors in $\U_r$ are diagonal, $V=\U_rQ_r^{1/2}$ consists of diagonal $N_F\times N_F$ blocks. Consequently, the actions of $V_1^{-1}$ and $V_1^{-*}$, as well as the construction and application of $W=V_1^{-*}V_2^*$, are performed modewise by elementwise operations on the corresponding diagonal entries. The resulting work and storage are both $\mathcal O(N_F)$. An exact explicit expression for $K_r^{-1}$ is
\begin{subequations}
\label{eq:Kr_inverse_explicit}
\begin{equation}
\label{eq:Kr_inverse_explicit_main}
K_r^{-1}
=
V_1^{-1}
\left[
M_1-(M_3-M_1W)S_M^{-1}(M_3^{*}-W^{*}M_1)
\right]
V_1^{-*},
\end{equation}
where
\begin{equation}
\label{eq:SM_def}
S_M =
M_2-M_3^{*}W-W^{*}M_3+W^{*}M_1W =
\begin{bmatrix}
-W^{*} & I
\end{bmatrix}
M
\begin{bmatrix}
-W\\
I
\end{bmatrix}
\succ0.
\end{equation}
\end{subequations}

To derive \eqref{eq:Kr_inverse_explicit_main}, consider $K_r\x=\y$ and set $\boldsymbol{\phi}=M^{-1}V\x$. Then $V^*\boldsymbol{\phi}=\y$ and $M\boldsymbol{\phi}=V\x$. Writing $\boldsymbol{\phi}=[\boldsymbol{\phi}_1^{\mathsf T},\boldsymbol{\phi}_2^{\mathsf T}]^{\mathsf T}$ gives
\begin{equation}
\label{eq:block_system_for_Kr_inverse}
\begin{cases}
M_1\boldsymbol{\phi}_1+M_3\boldsymbol{\phi}_2=V_1\x,\\
M_3^*\boldsymbol{\phi}_1+M_2\boldsymbol{\phi}_2=V_2\x,\\
V_1^*\boldsymbol{\phi}_1+V_2^*\boldsymbol{\phi}_2=\y.
\end{cases}
\end{equation}
With $\g=V_1^{-*}\y$, the third and first equations yield
\begin{equation}
\label{eq:phi_x_relations}
\boldsymbol{\phi}_1=\g-W\boldsymbol{\phi}_2,
\qquad
\x
=
V_1^{-1}
\left[
M_1\g+
(M_3-M_1W)\boldsymbol{\phi}_2
\right].
\end{equation}
Substitution into the second block equation gives
\begin{equation}
\label{eq:SM_phi2_equation}
S_M\boldsymbol{\phi}_2
=
(W^*M_1-M_3^*)\g.
\end{equation}
Solving for $\boldsymbol{\phi}_2$ and substituting into \eqref{eq:phi_x_relations} yields \eqref{eq:Kr_inverse_explicit_main}.

Although the explicit inverse formula \eqref{eq:Kr_inverse_explicit_main} involves the matrix $S_M$, the linear system associated with $S_M$ can be transformed into an equivalent system whose conditioning is controlled directly by the mass matrix $M$.

Let
\begin{equation}
\label{eq:Qtilde_def}
H=I+W^*W=R_H^*R_H,
\qquad
\widetilde Q=
\begin{bmatrix}
-W\\
I
\end{bmatrix}
R_H^{-1},
\end{equation}
where $R_H$ is the nonsingular Cholesky factor of $H$. Then
$$
\widetilde Q^*\widetilde Q
=
R_H^{-*}(I+W^*W)R_H^{-1}
=
I,
$$
so $\widetilde Q$ has orthonormal columns.

Define the orthogonally compressed mass matrix
\begin{equation}
\label{eq:Mhat_def}
\widehat{M} = \widetilde{Q}^{*}M\widetilde{Q}.
\end{equation}

Since $\begin{bmatrix}-W\\ I\end{bmatrix}=\widetilde Q R_H$, the matrix $S_M$ satisfies
\begin{equation}
\label{eq:SM_orthogonal_factorization}
S_M = R_H^*\widehat M R_H,
\qquad
S_M^{-1} = R_H^{-1}\widehat M^{-1}R_H^{-*}.
\end{equation}

Substituting \eqref{eq:SM_orthogonal_factorization} into \eqref{eq:Kr_inverse_explicit_main} yields
\begin{equation}
\label{eq:Kr_inverse_orthogonalized}
K_r^{-1} = V_1^{-1}M_1V_1^{-*} - B_r\widehat{M}^{-1}B_r^{*},
\end{equation}
where $B_r = V_1^{-1} \left( M_3-M_1W \right) R_H^{-1}$.
Therefore, each application of $K_r^{-1}$ requires the solution of a
Hermitian positive definite linear system with coefficient matrix
$\widehat{M}$, rather than $S_M$.

\begin{thm}
\label{thm:condition_number_Mhat}
Let $ \varepsilon_{\min,N} = \min_{\boldsymbol m}
\lambda_{\min}\!\left(\mathcal E(\x_{\boldsymbol m})\right), \;
\varepsilon_{\max,N} = \max_{\boldsymbol m}
\lambda_{\max}\!\left(\mathcal E(\x_{\boldsymbol m})\right)$.  Then the matrix $\widehat M$ defined in \eqref{eq:Mhat_def} is Hermitian positive definite and satisfies
\begin{equation}
\label{eq:Mhat_eigenvalue_bounds}
\varepsilon_{\min,N} =
\lambda_{\min}(M) \leq
\lambda_{\min}(\widehat M) \leq
\lambda_{\max}(\widehat M) \leq
\lambda_{\max}(M) =
\varepsilon_{\max,N}.
\end{equation}
Consequently,
\begin{equation}
\label{eq:Mhat_condition_bound}
\kappa_2(\widehat M)
\leq
\kappa_2(M)
=
\frac{\varepsilon_{\max,N}}{\varepsilon_{\min,N}}.
\end{equation}
Moreover, if
\begin{equation}
\label{eq:uniform_permittivity_bounds}
\varepsilon_{\min}I_3
\preceq
\mathcal E(\x)
\preceq
\varepsilon_{\max}I_3,
\qquad
\x\in\mathcal T^6,
\end{equation}
then
\begin{equation}
\label{eq:Mhat_permittivity_contrast_bound}
\kappa_2(\widehat M)
\leq
\kappa_2(M)
\leq
\frac{\varepsilon_{\max}}{\varepsilon_{\min}}.
\end{equation}
\end{thm}

\begin{proof}
By the construction of the Fourier-space mass matrix, $M=\mathscr F^*\mathcal D_{\mathcal E}\mathscr F$, where $\mathscr F$ is the normalized block discrete Fourier transform and
$\mathcal D_{\mathcal E}$ is permutation-similar to
$\operatorname{diag}_{\boldsymbol m}
\bigl(\mathcal E(\x_{\boldsymbol m})\bigr)$.
Since $\mathscr F$ is unitary,
\[
\lambda_{\min}(M)=\varepsilon_{\min,N},
\qquad
\lambda_{\max}(M)=\varepsilon_{\max,N}.
\]

Since $\widetilde Q^*\widetilde Q=I$, for every unit vector $\x$,
\[
\lambda_{\min}(M)
\leq
(\widetilde Q\x)^*M(\widetilde Q\x)
=
\x^*\widehat M\x
\leq
\lambda_{\max}(M).
\]
The Rayleigh-Ritz characterization gives
\eqref{eq:Mhat_eigenvalue_bounds} and
\eqref{eq:Mhat_condition_bound}.

Finally, \eqref{eq:uniform_permittivity_bounds} implies
$\varepsilon_{\min,N}\geq\varepsilon_{\min}$ and
$\varepsilon_{\max,N}\leq\varepsilon_{\max}$, which yields
\eqref{eq:Mhat_permittivity_contrast_bound}.
\end{proof}

Thus, the orthogonalized formulation eliminates the dependence of the inner CG conditioning on the nonorthogonal basis factor $[-W^{\mathsf T},I]^{\mathsf T}$. In particular, the transformed coefficient matrix satisfies
$\kappa_2(\widehat M)\leq\kappa_2(M)$ independently of the singular values of $W$.

In practical inverse Lanczos iterations, the inner system involving $\widehat M$ is solved by CG to a prescribed stopping tolerance, so the resulting inverse action is affected by the finite inner-solve residual. Motivated by perturbation analyses for Krylov and projection methods with approximate matrix-vector products \cite{SimonciniSzyld2003,Simoncini2005}, we derive a problem-specific estimate that separates the effects of the finite inner solves and full reorthogonalization on the inverse Ritz residual.

At the $j$-th Lanczos step, let $\widehat{\b}_j=B_r^*\z_j$ and let $\widehat{\c}_j$ be the exact solution of $\widehat M\widehat{\c}_j=\widehat{\b}_j$. Then
\begin{equation}
\label{eq:exact_inverse_action_inner}
K_r^{-1}\z_j
=
V_1^{-1}M_1V_1^{-*}\z_j
-
B_r\widehat{\c}_j.
\end{equation}
If CG produces $\widetilde{\c}_j$ with residual $\r_j=\widehat{\b}_j-\widehat M\widetilde{\c}_j$, then
$\widetilde{\c}_j-\widehat{\c}_j=-\widehat M^{-1}\r_j$. Hence the computed inverse action satisfies
\begin{equation}
\label{eq:cg_perturbed_inverse_action}
\widetilde{\p}_j
=
V_1^{-1}M_1V_1^{-*}\z_j
-
B_r\widetilde{\c}_j
=
K_r^{-1}\z_j+\f_j,
\qquad
\f_j=B_r\widehat M^{-1}\r_j.
\end{equation}
Consequently,
\begin{equation}
\label{eq:fj_bound}
\|\f_j\|_2
\leq
\frac{\|B_r\|_2}{\lambda_{\min}(\widehat M)}
\|\r_j\|_2.
\end{equation}
If CG is terminated when $\|\r_j\|_2/\|\widehat{\b}_j\|_2\leq\tau_{\mathrm{CG}}$, then $\|\widehat{\b}_j\|_2\leq\|B_r\|_2$ for $\|\z_j\|_2=1$, and therefore
\begin{equation}
\label{eq:eta_bound}
\|\f_j\|_2
\leq
\frac{\|B_r\|_2^2}
{\lambda_{\min}(\widehat M)}
\tau_{\mathrm{CG}}
\equiv\eta.
\end{equation}

For the following perturbation analysis, we consider one cycle of the fully reorthogonalized Lanczos iteration with finite-tolerance inverse actions. Collect the Lanczos vectors in $Z_m=[\z_1,\ldots,\z_m]$, with $Z_m^*Z_m=I_m$, and define $F_m=[\f_1,\ldots,\f_m]$. Let $T_m$ denote the nominal Lanczos tridiagonal matrix, and let $C_m\in\mathbb C^{m\times m}$ collect the additional coefficients removed by full reorthogonalization that are not represented in $T_m$. With $\e_m$ denoting the $m$th coordinate vector, the computed recurrence can be written as
\begin{equation}
\label{eq:perturbed_lanczos_recurrence}
K_r^{-1}Z_m+F_m = Z_m(T_m+C_m) + \beta_m\z_{m+1}\e_m^{\mathsf T}.
\end{equation}
For an exact Hermitian Lanczos process in exact arithmetic, the additional coefficients vanish. With independently terminated inner iterations, however, $C_m$ need not be zero.

\begin{thm}
\label{thm:inverse_lanczos_ritz_bound}
Let $T_m\s=\theta\s$, $\|\s\|_2=1$, and $\v=Z_m\s$. Then
\begin{equation}
\label{eq:ritz_residual_bound}
\|K_r^{-1}\v-\theta\v\|_2 \leq
r_m^{\mathrm{Lan}} + \|C_m\s\|_2 + \sqrt{m}\,\eta,
\end{equation}
where $ r_m^{\mathrm{Lan}} = |\beta_m\e_m^{\mathsf T}\s| $ is the nominal three-term recurrence residual, and $\eta$ is defined in \eqref{eq:eta_bound}.
\end{thm}

\begin{proof}
Multiplying \eqref{eq:perturbed_lanczos_recurrence} by $\s$ and using $T_m\s=\theta\s$ and $\v=Z_m\s$ give
$$
K_r^{-1}\v-\theta\v = Z_mC_m\s + \beta_m\z_{m+1}\e_m^{\mathsf T}\s - F_m\s.
$$
Since $Z_m^*Z_m=I_m$ and $\|\z_{m+1}\|_2=1$, the triangle inequality yields
\begin{align*}
\|K_r^{-1}\v-\theta\v\|_2
&\leq
\|C_m\s\|_2 + |\beta_m\e_m^{\mathsf T}\s| + \|F_m\s\|_2 \\
&\leq
\|C_m\s\|_2 + r_m^{\mathrm{Lan}} + \|F_m\|_F .
\end{align*}
Finally, \eqref{eq:eta_bound} implies $\|F_m\|_F = \left( \sum_{j=1}^{m}\|\f_j\|_2^2 \right)^{1/2} \leq \sqrt{m}\,\eta$, which proves \eqref{eq:ritz_residual_bound}.
\end{proof}
The implementation monitors $r_m^{\mathrm{Lan}}$ as a stopping indicator. The additional terms in \eqref{eq:ritz_residual_bound} are not included in this indicator.

The inverse representation \eqref{eq:Kr_inverse_orthogonalized} is implemented in operator form, without explicitly forming $V_1^{-1}$, $V_1^{-*}$, or $\widehat M^{-1}$. The complete procedure for computing the smallest positive Maxwell eigenpairs is summarized in Algorithm~\ref{alg:inverse_lanczos}.

\begin{algorithm}[htbp]
\caption{Inverse Lanczos method for the null-space free EVP
\eqref{eq:null_space_free_EVP}}
\label{alg:inverse_lanczos}
\begin{algorithmic}[1]

\Require Mass matrix $M$, transverse basis $\U_r$, positive diagonal matrix $Q_r$; number of desired eigenpairs $n_{\mathrm{ev}}$; nonzero initial vector $\z_1$; tolerances $\tau_{\mathrm{CG}}$ and $\tau_{\mathrm{Lan}}$, Krylov dimension $m_{\max}$, and iteration/restart limits.

\Ensure The $n_{\mathrm{ev}}$ smallest positive eigenvalues $\widetilde{\lambda}_i$, the corresponding reduced eigenvectors $\v_i$, and Fourier eigenvectors $\U_i$, $i=1,\ldots,n_{\mathrm{ev}}$.

\State Form $V=\U_rQ_r^{1/2}$ and partition $V=[V_1^{\mathsf T},V_2^{\mathsf T}]^{\mathsf T}$ and $M=\begin{bmatrix}M_1&M_3\\M_3^*&M_2\end{bmatrix}$ as in \eqref{eq:V_M_block_def}, with $V_1$ nonsingular.

\State Set $W=V_1^{-*}V_2^*$ and construct $R_H$, $B_r$, and $\widehat M$ using \eqref{eq:Qtilde_def}-\eqref{eq:Kr_inverse_orthogonalized}.

\State Normalize $\z_1$ and apply fully reorthogonalized inverse Lanczos using \eqref{eq:Kr_inverse_orthogonalized}, solving the $\widehat M$-systems by CG to tolerance $\tau_{\mathrm{CG}}$. Select the $n_{\mathrm{ev}}$ largest Ritz pairs $(\theta_i,\v_i)$ using the stated recurrence-based stopping criterion and iteration limits.

\State Set $\widetilde{\lambda}_i=\theta_i^{-1}$, $i=1,\ldots,n_{\mathrm{ev}}$.

\For{$i=1,\ldots,n_{\mathrm{ev}}$}
\State Solve $M\U_i=\U_rQ_r^{1/2}\v_i$ and normalize $\U_i^*M\U_i=1$.
\EndFor

\end{algorithmic}
\end{algorithm}

The inverse Lanczos procedure therefore produces approximations to the positive Fourier eigenpairs of the original discrete Maxwell problem. These eigenpairs are still represented in the 6D Fourier discretization. To interpret the computed modes on the physical quasiperiodic cut and to assess their consistency independently of the Fourier discretization, we next reconstruct the corresponding fields on a 3D staggered Yee grid and introduce cropped Yee residual and local Rayleigh-quotient measures for assessing physical-space consistency.

\section{Physical-Space Reconstruction and Local Rayleigh Quotient Analysis}
\label{sec:LWRQ_error_analysis}

The null-space free inverse Lanczos method yields a positive eigenpair $(\widetilde{\lambda},\U)$ of the Fourier-discretized Maxwell problem \eqref{eq:discrete_gevp}. Restriction of $\U$ to the physical cut induces a 3D quasiperiodic field, whereas both $\U$ and $\widetilde{\lambda}$ arise from the 6D Fourier discretization. We therefore reconstruct this field on a 3D staggered Yee grid and evaluate the associated cropped Yee residual and Rayleigh quotient. LWRQs are further combined by mass-weighted averaging to obtain a scalar spectral estimate.

\subsection{Yee operators and cropped sampling windows}

We use the standard staggered Yee discretization~\cite{Yee1966} on a uniform mesh of size $h$. Let $\mathcal C_j$ denote the forward difference in the $j$th coordinate direction and $\mathcal C_j^\sharp$ its formal adjoint. The discrete curl and curl-curl operators are
\begin{equation}
\mathcal C
=
\begin{pmatrix}
0 & -\mathcal C_3 & \mathcal C_2\\
\mathcal C_3 & 0 & -\mathcal C_1\\
-\mathcal C_2 & \mathcal C_1 & 0
\end{pmatrix},
\qquad
\mathcal A=\mathcal C^\sharp\mathcal C.
\label{eq:A_CTC_block}
\end{equation}
Thus, $\mathcal A$ is the usual Yee curl-curl stencil acting on the three staggered electric-field components. For the Yee postprocessing, we restrict the permittivity to tensors diagonal in the physical coordinate axes, including all media used below. The mass operator $\mathcal B$ acts componentwise by multiplication with $\varepsilon_{\ell\ell}(\r_{\boldsymbol\zeta}^{(\ell)})$ at the corresponding Yee point, and its diagonal entries are uniformly positive.

For the physical-space consistency analysis, let $\Omega_n=\{\boldsymbol{\zeta}=(i,j,k)\in\mathbb Z^3:-n\leq i,j,k\leq n\}$ denote the interior sampling indices. Since the curl-curl stencil on $\Omega_n$ requires neighboring Yee values, the field is reconstructed on the enlarged set $\Omega_n^{+}=\{\boldsymbol{\zeta}\in\mathbb Z^3:\operatorname{dist}_{\infty}(\boldsymbol{\zeta},\Omega_n)\leq1\}$.

If $\widehat{\u}$ denotes the reconstructed Yee field on $\Omega_n^{+}$ and $\widehat{\u}_{\mathrm{o}}$ its restriction to $\Omega_n$, the cropped operator actions are defined by
\begin{subequations}
\label{total:finite_yee_matrix_derive}
\begin{align}
\widehat A\widehat{\u} &= \mathcal A_{\{\Omega_n,\Omega_n^{+}\}}\widehat{\u},
\label{eq:cropped_A_matrix}
\\
\widehat B\widehat{\u}_{\mathrm{o}} &= \mathcal B_{\{\Omega_n,\Omega_n\}}\widehat{\u}_{\mathrm{o}}.
\label{eq:cropped_B_matrix}
\end{align}
\end{subequations}
These cropped operators are used below only for the physical-space Rayleigh-quotient and residual computations.

\subsection{Memory-efficient reconstruction of the 3D Yee field}
\label{subsec:memory_efficient_Yee_reconstruction}

For $\boldsymbol{\xi}\in\mathcal J_v$, recall that
$\q_{\boldsymbol{\xi}}=P^{\mathsf T}(\k+\boldsymbol{\xi})$.
For a Fourier eigenvector, the $\ell$th component of the corresponding physical field at a Yee location is
\begin{equation}
\label{eq:direct_Yee_field_reconstruction}
\widehat{u}^{(\ell)}
\left(
\r_{\boldsymbol{\zeta}}^{(\ell)}
\right)
=
\sum_{\boldsymbol{\xi}\in\mathcal{J}_v}
U_{\ell,\boldsymbol{\xi}}
\exp
\left( \imath \q_{\boldsymbol{\xi}}^{\mathsf{T}} \r_{\boldsymbol{\zeta}}^{(\ell)}
\right), \qquad \ell=1,2,3,
\end{equation}
where $x_i=ih$, $y_j=jh$, and $z_k=kh$, and a hat over an index denotes a half-grid shift in the corresponding coordinate direction, e.g., $x_{\hat{i}}=(i+\frac12)h$. Thus, $\r_{\boldsymbol{\zeta}}^{(1)}=(x_{\hat{i}},y_j,z_k)$, $\r_{\boldsymbol{\zeta}}^{(2)}=(x_i,y_{\hat{j}},z_k)$, and $\r_{\boldsymbol{\zeta}}^{(3)}=(x_i,y_j,z_{\hat{k}})$.

Let $N_g=\#\Omega_n^{+}$ and $N_F=\#\mathcal{J}_v$. A direct simultaneous evaluation of \eqref{eq:direct_Yee_field_reconstruction} would require a dense phase matrix
\begin{equation}
\label{eq:full_phase_matrix}
\Phi^{(\ell)} =
\left[
\exp
\left( \imath \q_{\boldsymbol{\xi}}^{\mathsf{T}} \r_{\boldsymbol{\zeta}}^{(\ell)}
\right)
\right]_{
\boldsymbol{\zeta}\in\Omega_n^{+}, \,
\boldsymbol{\xi}\in\mathcal{J}_v
}
\in
\mathbb{C}^{N_g\times N_F}.
\end{equation}
The storage requirement of \eqref{eq:full_phase_matrix} is $\mathcal O(N_gN_F)$, which becomes prohibitive when both the 3D sampling region and the 6D Fourier index set are large. To avoid forming this dense phase matrix, we exploit the local phase structure of the projected Fourier expansion and construct a separated Taylor representation.

Specifically, we partition $\Omega_n^+$ into $S$ pairwise disjoint subsets $\Omega_s$ and assign a center $\r_{0,s}^{(\ell)}$ to each component on each subset. In the following local formulas, we suppress the subset index, write $\r_0=\r_{0,s}^{(\ell)}$, and set $\boldsymbol{\delta}_{\boldsymbol{\zeta}}^{(\ell)}=\r_{\boldsymbol{\zeta}}^{(\ell)}-\r_0$ for $\boldsymbol{\zeta}\in\Omega_s$. For a multi-index $\boldsymbol{\alpha}=(\alpha_1,\alpha_2,\alpha_3)\in\mathbb N^3$, where $\mathbb N=\{0,1,2,\ldots\}$, write $|\boldsymbol{\alpha}|=\alpha_1+\alpha_2+\alpha_3$, $\boldsymbol{\alpha}!=\alpha_1!\alpha_2!\alpha_3!$, $\q_{\boldsymbol{\xi}}^{\boldsymbol{\alpha}}=\prod_{s=1}^3q_{\boldsymbol{\xi},s}^{\alpha_s}$, and $(\boldsymbol{\delta}_{\boldsymbol{\zeta}}^{(\ell)})^{\boldsymbol{\alpha}}=\prod_{s=1}^3(\delta_{\boldsymbol{\zeta},s}^{(\ell)})^{\alpha_s}$.

Truncating the multivariate Taylor expansion at total degree $p$ gives
\begin{equation}
\label{eq:multivariate_Taylor_exponential}
\exp
\left(
\imath \q_{\boldsymbol{\xi}}^{\mathsf{T}} \boldsymbol{\delta}_{\boldsymbol{\zeta}}^{(\ell)}
\right)
\approx
\sum_{|\boldsymbol{\alpha}|\leq p}
\frac{
\imath^{|\boldsymbol{\alpha}|}
}{
\boldsymbol{\alpha}!
}
\q_{\boldsymbol{\xi}}^{\boldsymbol{\alpha}}
\left(
\boldsymbol{\delta}_{\boldsymbol{\zeta}}^{(\ell)}
\right)^{\boldsymbol{\alpha}}.
\end{equation}
For each component and multi-index, define the Fourier moment
\begin{equation}
\label{eq:Fourier_moment_Taylor_reconstruction}
\gamma_{\boldsymbol{\alpha}}^{(\ell)}
=
\sum_{\boldsymbol{\xi}\in\mathcal{J}_v}
U_{\ell,\boldsymbol{\xi}}
\exp
\left(\imath \q_{\boldsymbol{\xi}}^{\mathsf{T}} \r_0
\right)
\q_{\boldsymbol{\xi}}^{\boldsymbol{\alpha}}.
\end{equation}
Substituting \eqref{eq:multivariate_Taylor_exponential} into \eqref{eq:direct_Yee_field_reconstruction} yields
\begin{equation}
\label{eq:Taylor_Yee_field_reconstruction}
\widehat{u}^{(\ell)}
\left(
\r_{\boldsymbol{\zeta}}^{(\ell)}
\right)
\approx
\sum_{|\boldsymbol{\alpha}|\leq p}
\frac{
\imath^{|\boldsymbol{\alpha}|}
}{
\boldsymbol{\alpha}!
}
\left(
\boldsymbol{\delta}_{\boldsymbol{\zeta}}^{(\ell)}
\right)^{\boldsymbol{\alpha}}
\gamma_{\boldsymbol{\alpha}}^{(\ell)},
\qquad
\boldsymbol{\zeta}\in\Omega_s.
\end{equation}
Let $ L_p=\#\{\boldsymbol{\alpha}\in\mathbb N^3:|\boldsymbol{\alpha}|\le p\} =\binom{p+3}{3}$.

Define the maximum centered phase radius for the $\ell$-th component by
\begin{equation}
\label{eq:Taylor_phase_radius}
r_{\mathrm{ph}}^{(\ell)} =
\max_{
\substack{
\boldsymbol{\zeta}\in\Omega_s\\
\boldsymbol{\xi}\in\mathcal{J}_v } }
\left|
\q_{\boldsymbol{\xi}}^{\mathsf{T}} \boldsymbol{\delta}_{\boldsymbol{\zeta}}^{(\ell)}
\right|.
\end{equation}
Using the standard Taylor remainder for $e^{\imath t}$ with real $t$ and the triangle inequality, the field-reconstruction error satisfies
\begin{equation}
\left|
\widehat{u}^{(\ell)}
\left( \r_{\boldsymbol{\zeta}}^{(\ell)} \right) -
\sum_{|\boldsymbol{\alpha}|\leq p}
\frac{ \imath^{|\boldsymbol{\alpha}|} }{ \boldsymbol{\alpha}! }
\left( \boldsymbol{\delta}_{\boldsymbol{\zeta}}^{(\ell)} \right)^{\boldsymbol{\alpha}}
\gamma_{\boldsymbol{\alpha}}^{(\ell)}
\right|
\leq
\frac{
\left( r_{\mathrm{ph}}^{(\ell)} \right)^{p+1}
}{ (p+1)! }
\sum_{\boldsymbol{\xi}\in\mathcal{J}_v}
\left|
U_{\ell,\boldsymbol{\xi}}
\right|.
\label{eq:Taylor_Yee_reconstruction_error}
\end{equation}

Let $r_{\mathrm{ph},\max} =\max_{\ell,s}r_{\mathrm{ph},s}^{(\ell)}$, where $r_{\mathrm{ph},s}^{(\ell)}$ denotes the local radius in \eqref{eq:Taylor_phase_radius} on $\Omega_s$. Reducing the local radii tightens the error bound \eqref{eq:Taylor_Yee_reconstruction_error}. For one Fourier field, forming the moments at all centers and evaluating the local polynomials require $\mathcal O(SL_pN_F+L_pN_g)$ operations. The dense $N_g\times N_F$ phase matrix is avoided; working storage depends on batching and the cached moment and monomial tables.

\subsection{LWRQs and physical-space consistency}

Weighted Rayleigh quotients for reconstructed scalar quasiperiodic fields were introduced in \cite{SunLiLinLyu2026}. The present Maxwell setting requires a different construction because the three staggered field components are coupled by the curl-curl operator, and direct componentwise quotients may be ill-conditioned when individual components are locally small. We therefore average the local energy and mass densities separately before forming their ratio.

The Rayleigh quotient of the cropped Yee system is
\begin{equation}
\label{eq:cropped_global_RQ_def}
R_{\Omega_n} \left( \widehat{\u} \right) =
\frac{ \operatorname{Re}\!\left(\widehat{\u}_{\mathrm{o}}^{*} \widehat{A}\widehat{\u} \right) }
{ \widehat{\u}_{\mathrm{o}}^{*} \widehat{B}\widehat{\u}_{\mathrm{o}} },
\qquad
\widehat{\u}_{\mathrm{o}}^{*} \widehat{B}\widehat{\u}_{\mathrm{o}} >0.
\end{equation}
For each $\boldsymbol\eta\in\Omega_n$, let $\widehat B_{\boldsymbol\eta}\in\mathbb R^{3\times3}$ be the positive diagonal mass block defined by these componentwise samples, so that $(\widehat B\widehat{\u}_{\mathrm{o}})_{\boldsymbol\eta} =\widehat B_{\boldsymbol\eta}\widehat{\u}_{\boldsymbol\eta}$.

Let $N(\boldsymbol{\zeta})\subseteq\Omega_n$ be a neighborhood of $\boldsymbol{\zeta}$. We extend every weight by zero outside its neighborhood $ w_{\boldsymbol{\zeta}\boldsymbol{\eta}} = 0, \boldsymbol{\eta}\notin N(\boldsymbol{\zeta})$. We define the weighted energy and mass by
\begin{equation}
\label{eq:weighted_energy_mass_density}
E_{\boldsymbol{\zeta}} =
\sum_{\boldsymbol{\eta}\in\Omega_n} w_{\boldsymbol{\zeta}\boldsymbol{\eta}} \operatorname{Re}
\left[\widehat{\u}_{\boldsymbol{\eta}}^{*} \left( \widehat{A}\widehat{\u} \right)_{\boldsymbol{\eta}}\right],
\qquad
M_{\boldsymbol{\zeta}} = 
\sum_{\boldsymbol{\eta}\in\Omega_n} w_{\boldsymbol{\zeta}\boldsymbol{\eta}} \widehat{\u}_{\boldsymbol{\eta}}^{*}
\left( \widehat{B}\widehat{\u}_{\mathrm{o}} \right)_{\boldsymbol{\eta}}.
\end{equation}
Throughout the LWRQ analysis and postprocessing, we assume $M_{\boldsymbol\zeta}>0$ for every $\boldsymbol\zeta\in\Omega_n$ and define the LWRQ by
\begin{equation}
\label{eq:weighted_local_RQ_def}
\rho_{\boldsymbol{\zeta}} =
\frac{
E_{\boldsymbol{\zeta}}
}{
M_{\boldsymbol{\zeta}}
}.
\end{equation}

Let $\r_{\boldsymbol{\zeta}}$ denote a representative cell coordinate. Define the unnormalized Gaussian weights by
$$
\widetilde w_{\boldsymbol{\zeta}\boldsymbol{\eta}}
=
\begin{cases}
\exp\!\left(
-\|\r_{\boldsymbol{\zeta}}-\r_{\boldsymbol{\eta}}\|_2^2/\sigma^2
\right),
& \boldsymbol{\eta}\in N(\boldsymbol{\zeta}),\\
0,
& \text{otherwise}.
\end{cases}
$$
Assuming the normalization denominator is nonzero, define
\begin{equation}
w_{\boldsymbol{\zeta}\boldsymbol{\eta}}
=
\frac{\widetilde{w}_{\boldsymbol{\zeta}\boldsymbol{\eta}}}
{\sum_{\boldsymbol{\zeta}'\in\Omega_n}
\widetilde{w}_{\boldsymbol{\zeta}'\boldsymbol{\eta}}},
\qquad
\sum_{\boldsymbol{\zeta}\in\Omega_n}
w_{\boldsymbol{\zeta}\boldsymbol{\eta}}=1,
\qquad  \boldsymbol{\eta}\in\Omega_n
\label{eq:partition_unity_weight}
\end{equation}
The mass-weighted probability and the corresponding LWRQ expectation are
\begin{equation}
p_{\boldsymbol{\zeta}} =
\frac{M_{\boldsymbol{\zeta}}}
{\sum_{\boldsymbol{\eta}\in\Omega_n}M_{\boldsymbol{\eta}}},
\qquad
\widehat{\lambda}_{\mathrm{LWRQ}} =
\sum_{\boldsymbol{\zeta}\in\Omega_n}
p_{\boldsymbol{\zeta}}\rho_{\boldsymbol{\zeta}}.
\label{eq:weighted_expectation_local_RQ}
\end{equation}

\begin{thm}
\label{thm:weighted_local_RQ_to_global_RQ}
Suppose that the nonnegative weights satisfy \eqref{eq:partition_unity_weight} and that $M_{\boldsymbol\zeta}>0$ for all $\boldsymbol\zeta\in\Omega_n$. Then the weighted expectation in \eqref{eq:weighted_expectation_local_RQ} coincides with the Rayleigh quotient of the cropped Yee system:
\begin{equation}
\label{eq:weighted_expectation_equals_global_RQ}
\mathbb{E}_{\mathrm{p}}
\left[ \left\{ \rho_{\boldsymbol{\zeta}} \right\}_{\boldsymbol{\zeta}\in\Omega_n} \right]
=
R_{\Omega_n} \left( \widehat{\u} \right)
=
\frac{ \operatorname{Re}\! \left(\widehat{\u}_{\mathrm{o}}^{*} \widehat{A}\widehat{\u} \right)}
{ \widehat{\u}_{\mathrm{o}}^{*} \widehat{B}\widehat{\u}_{\mathrm{o}}}.
\end{equation}
\end{thm}

\begin{proof}
Using \eqref{eq:weighted_energy_mass_density}, \eqref{eq:weighted_local_RQ_def}, and \eqref{eq:weighted_expectation_local_RQ}, we obtain
\begin{equation*}
\widehat{\lambda}_{\mathrm{LWRQ}} =
\frac{ \sum_{\boldsymbol{\zeta}\in\Omega_n} E_{\boldsymbol{\zeta}} }
{ \sum_{\boldsymbol{\zeta}\in\Omega_n} M_{\boldsymbol{\zeta}} } =
\frac{ \sum_{\boldsymbol{\eta}\in\Omega_n}
\left( \sum_{\boldsymbol{\zeta}\in\Omega_n} w_{\boldsymbol{\zeta}\boldsymbol{\eta}} \right)
\operatorname{Re} 
\left[ \widehat{\u}_{\boldsymbol{\eta}}^* (\widehat A\widehat{\u})_{\boldsymbol{\eta}} \right] }
{ \sum_{\boldsymbol{\eta}\in\Omega_n} \left( \sum_{\boldsymbol{\zeta}\in\Omega_n} w_{\boldsymbol{\zeta}\boldsymbol{\eta}} \right)
\widehat{\u}_{\boldsymbol{\eta}}^* (\widehat B\widehat{\u}_{\mathrm{o}})_{\boldsymbol{\eta}} }.
\end{equation*}
Using \eqref{eq:partition_unity_weight} gives
\begin{equation*}
\widehat{\lambda}_{\mathrm{LWRQ}}
=\frac{\operatorname{Re}
(\widehat{\u}_{\mathrm{o}}^*\widehat A\widehat{\u})}
{\widehat{\u}_{\mathrm{o}}^*\widehat B
\widehat{\u}_{\mathrm{o}}}
=R_{\Omega_n}(\widehat{\u}).
\end{equation*}
This proves the identity.
\end{proof}

\begin{thm}
\label{thm:weighted_local_RQ_smoothness}
Let $\boldsymbol{\zeta},\boldsymbol{\zeta}^{\prime}\in\Omega_n$ and define
\begin{equation}
\label{eq:weight_variation_measure}
\omega_w \left( \boldsymbol{\zeta}, \boldsymbol{\zeta}^{\prime} \right) 
=
\sum_{\boldsymbol{\eta}\in\Omega_n}
\left|
w_{\boldsymbol{\zeta}^{\prime}\boldsymbol{\eta}} - w_{\boldsymbol{\zeta}\boldsymbol{\eta}}
\right|.
\end{equation}
Suppose that
$$
\begin{aligned}
\left|
\operatorname{Re}\!\left[
\widehat{\u}_{\boldsymbol{\eta}}^{*}
(\widehat A\widehat{\u})_{\boldsymbol{\eta}}
\right]
\right|12
&\leq C_a,
&
\widehat{\u}_{\boldsymbol{\eta}}^{*}
\widehat B_{\boldsymbol{\eta}}
\widehat{\u}_{\boldsymbol{\eta}}
&\leq C_b,
\qquad
\boldsymbol{\eta}\in\Omega_n,
\sup_{\boldsymbol{\xi}\in\Omega_n}
&
\sum_{\boldsymbol{\eta}\in\Omega_n}
 w_{\boldsymbol{\xi}\boldsymbol{\eta}} \leq C_0.
\end{aligned}
$$
If
\begin{equation}
\label{eq:local_mass_positive_lower_bound}
M_{\boldsymbol{\zeta}}\geq c>0,
\qquad
M_{\boldsymbol{\zeta}^{\prime}}\geq c>0,
\end{equation}
then
\begin{equation}
\label{eq:weighted_local_RQ_smoothness_bound}
\left|
\rho_{\boldsymbol{\zeta}^{\prime}}
-
\rho_{\boldsymbol{\zeta}}
\right|
\leq
C\,
\omega_w
\left(
\boldsymbol{\zeta},
\boldsymbol{\zeta}^{\prime}
\right),
\end{equation}
where $C>0$ depends only on $C_a$, $C_b$, $C_0$, and $c$. Consequently, if
$\omega_w(\boldsymbol{\zeta},\boldsymbol{\zeta}^{\prime})\leq C_wh$,
then
$|\rho_{\boldsymbol{\zeta}^{\prime}}-\rho_{\boldsymbol{\zeta}}|\leq Ch$.
\end{thm}

\begin{proof}
Let $ a_{\boldsymbol{\eta}} =
\operatorname{Re}\!\left[
\widehat{\u}_{\boldsymbol{\eta}}^{*}
(\widehat A\widehat{\u})_{\boldsymbol{\eta}}
\right]$, $
b_{\boldsymbol{\eta}} =
\widehat{\u}_{\boldsymbol{\eta}}^{*}
\widehat B_{\boldsymbol{\eta}}
\widehat{\u}_{\boldsymbol{\eta}}$. Then $
E_{\boldsymbol{\zeta}} =
\sum_{\boldsymbol{\eta}\in\Omega_n}
w_{\boldsymbol{\zeta}\boldsymbol{\eta}}
a_{\boldsymbol{\eta}} $, $
M_{\boldsymbol{\zeta}} =
\sum_{\boldsymbol{\eta}\in\Omega_n}
w_{\boldsymbol{\zeta}\boldsymbol{\eta}}
b_{\boldsymbol{\eta}}$.  Hence,
$$
|E_{\boldsymbol{\zeta}^{\prime}}-E_{\boldsymbol{\zeta}}|
\leq
C_a\,
\omega_w(\boldsymbol{\zeta},\boldsymbol{\zeta}^{\prime}),
\qquad
|M_{\boldsymbol{\zeta}^{\prime}}-M_{\boldsymbol{\zeta}}|
\leq
C_b\,
\omega_w(\boldsymbol{\zeta},\boldsymbol{\zeta}^{\prime}).
$$
Moreover,
$|E_{\boldsymbol{\zeta}}|\leq C_aC_0$.
Using
$$
\rho_{\boldsymbol{\zeta}^{\prime}}
-
\rho_{\boldsymbol{\zeta}}
=
\frac{
(E_{\boldsymbol{\zeta}^{\prime}}-E_{\boldsymbol{\zeta}})
M_{\boldsymbol{\zeta}}
-
E_{\boldsymbol{\zeta}}
(M_{\boldsymbol{\zeta}^{\prime}}-M_{\boldsymbol{\zeta}})
}{
M_{\boldsymbol{\zeta}^{\prime}}
M_{\boldsymbol{\zeta}}
},
$$
together with \eqref{eq:local_mass_positive_lower_bound}, yields
\eqref{eq:weighted_local_RQ_smoothness_bound}.
\end{proof}

The following estimate relates the cropped Yee quotient to the residual of the reconstructed field.

\begin{thm}
\label{thm:error_bound_weighted_LRQ}
Under the assumptions of Theorem~\ref{thm:weighted_local_RQ_to_global_RQ}, assume that, for some $\vartheta\in\mathbb R$, the reconstructed field satisfies
\begin{equation}
\label{eq:residual_relation_weighted}
\widehat A\widehat{\u} =
\vartheta \widehat B\widehat{\u}_{\mathrm{o}} + \e.
\end{equation}
Suppose that $e_{\max} = \max_{\boldsymbol{\zeta}\in\Omega_n} \|\e_{\boldsymbol{\zeta}}\|_{\infty}$, $\sum_{\boldsymbol{\zeta}\in\Omega_n} \|\widehat{\u}_{\boldsymbol{\zeta}}\|_1 \leq C_u$, and $\widehat{\u}_{\mathrm{o}}^{*} \widehat B\widehat{\u}_{\mathrm{o}} \geq c_B>0$. Then
\begin{equation}
\label{eq:error_expectation_bound}
\left| \mathbb E_{\mathrm p}
\left[
\left\{ \rho_{\boldsymbol{\zeta}} \right\}_{\boldsymbol{\zeta}\in\Omega_n}
\right] - \vartheta
\right|
\leq \frac{C_u}{c_B}e_{\max}.
\end{equation}
In particular, for $\vartheta=\widetilde{\lambda}$, $ \left|
\widehat{\lambda}_{\mathrm{LWRQ}} - \widetilde{\lambda}
\right| \leq
\frac{C_u}{c_B}e_{\max}$. 
\end{thm}

\begin{proof}
By Theorem~\ref{thm:weighted_local_RQ_to_global_RQ} and \eqref{eq:residual_relation_weighted},
\begin{align*}
\mathbb E_{\mathrm p}
\left[
\left\{ \rho_{\boldsymbol{\zeta}} \right\}_{\boldsymbol{\zeta}\in\Omega_n}
\right] -\vartheta =
R_{\Omega_n}(\widehat{\u})-\vartheta =
\frac{ \operatorname{Re} (\widehat{\u}_{\mathrm{o}}^*\e) }
{ \widehat{\u}_{\mathrm{o}}^* \widehat B\widehat{\u}_{\mathrm{o}} }.
\end{align*}
Therefore,
\begin{equation*}
\left|\mathbb E_{\mathrm p}
\left[\left\{\rho_{\boldsymbol\zeta}\right\}\right]
-\vartheta\right|
\leq
\frac{\sum_{\boldsymbol\zeta\in\Omega_n}
\|\widehat{\u}_{\boldsymbol\zeta}\|_1
\|\e_{\boldsymbol\zeta}\|_\infty}{c_B}
\leq\frac{C_u}{c_B}e_{\max}.
\end{equation*}
This proves \eqref{eq:error_expectation_bound}.
\end{proof}

\begin{remark}
\label{remark:Taylor_error_in_residual}
Let $\delta\widehat{\u} = \widehat{\u}_{T} - \widehat{\u}_{\mathrm{dir}}$ denote the Taylor reconstruction error. If the directly reconstructed field satisfies
\begin{equation}
\widehat A\widehat{\u}_{\mathrm{dir}} =
\vartheta\widehat B\widehat{\u}_{\mathrm{dir,o}} + \e_{\mathrm{disc}},
\end{equation}
then the residual of the Taylor-reconstructed field can be decomposed as
\begin{equation}
\e = \e_{\mathrm{disc}} + \e_{\mathrm{rec}},
\qquad
\e_{\mathrm{rec}} =
\widehat A\delta\widehat{\u} - \vartheta\widehat B\delta\widehat{\u}_{\mathrm{o}}.
\end{equation}
Consequently,
\begin{equation}
e_{\max} \leq
e_{\mathrm{disc,max}} + e_{\mathrm{rec,max}},
\qquad
e_{\mathrm{rec,max}} \leq
\left( \|\widehat A\|_{\infty} + |\vartheta|\|\widehat B\|_{\infty} \right)
\|\delta\widehat{\u}\|_{\infty}.
\end{equation}
Thus, the reconstruction contribution is controlled by \eqref{eq:Taylor_Yee_reconstruction_error}.
\end{remark}

For clarity, the physical-space reconstruction and LWRQ postprocessing developed above are summarized in Algorithm~\ref{alg:physical_reconstruction_LWRQ}.

\begin{algorithm}[htbp]
\caption{Physical-space reconstruction and LWRQ recovery}
\label{alg:physical_reconstruction_LWRQ}
\begin{algorithmic}[1]

\Require Fourier eigenvector $\U$, projection matrix $P$, Bloch vector $\k$, and retained index set $\mathcal J_v$; Yee spacing $h$ and half-width $n$; Taylor order $p$, partition $\{\Omega_s\}_{s=1}^S$ of $\Omega_n^+$ and associated componentwise centers; lifted permittivity $\mathcal E$, neighborhoods $N(\boldsymbol\zeta)$, and Gaussian width $\sigma$.

\Ensure Reconstructed Yee field $\widehat{\u}$,
cropped Rayleigh quotient $R_{\Omega_n}(\widehat{\u})$,
local quotients $\rho_{\boldsymbol\zeta}$, and
LWRQ approximation $\widehat{\lambda}_{\mathrm{LWRQ}}$.

\State Form
$\q_{\boldsymbol\xi}=P^{\mathsf T}(\k+\boldsymbol\xi)$,
$\boldsymbol\xi\in\mathcal J_v$, and generate the staggered Yee points
on $\Omega_n^+$.

\State Compute the local Fourier moments
$\gamma_{\boldsymbol\alpha,s}^{(\ell)}$,
$|\boldsymbol\alpha|\leq p$, from
\eqref{eq:Fourier_moment_Taylor_reconstruction}, and reconstruct
$\widehat{\u}$ by \eqref{eq:Taylor_Yee_field_reconstruction}.

\State Restrict $\widehat{\u}$ to $\widehat{\u}_{\mathrm{o}}$ on
$\Omega_n$, apply the cropped Yee operators, and compute
$$
R_{\Omega_n}(\widehat{\u})
=
\frac{\operatorname{Re}
\bigl(\widehat{\u}_{\mathrm{o}}^*\widehat A\widehat{\u}\bigr)}
{\widehat{\u}_{\mathrm{o}}^*\widehat B\widehat{\u}_{\mathrm{o}}}.
$$

\State Form the weights using \eqref{eq:partition_unity_weight}
and compute $E_{\boldsymbol\zeta}$ and
$M_{\boldsymbol\zeta}$ from
\eqref{eq:weighted_energy_mass_density}.
Under the standing positive-mass assumption, compute
$\rho_{\boldsymbol\zeta}$ from
\eqref{eq:weighted_local_RQ_def}.

\State Set
$p_{\boldsymbol\zeta} = M_{\boldsymbol\zeta}/\sum_{\boldsymbol\eta}M_{\boldsymbol\eta}$
and compute $ \widehat{\lambda}_{\mathrm{LWRQ}} = \sum_{\boldsymbol\zeta} p_{\boldsymbol\zeta}\rho_{\boldsymbol\zeta}$
according to \eqref{eq:weighted_expectation_local_RQ}.

\end{algorithmic}
\end{algorithm}

The numerical experiments below examine the reconstruction accuracy and the resulting 3D Yee residual and spectral estimates.

\section{Numerical experiments}
\label{sec:numerical_experiments}

In this section, we validate the null-space free reduction and inverse Lanczos solver, the orthogonalized inner system, the memory-efficient Yee-field reconstruction, and the cropped Yee and LWRQ consistency measures, and then apply the complete framework to a representative 3D quasiperiodic Maxwell problem. Unless otherwise stated, all computations are implemented in MATLAB R2026a using double-precision arithmetic with $\mu(\r)=I_3$, and the iterative eigensolvers target the smallest positive eigenvalues. The main computations are performed in double precision on an NVIDIA Tesla V100-SXM2 GPU with 32~GB of device memory.

\subsection{Experiment 1: Large-scale comparison of original and reduced inverse realizations}
\label{subsec:exp_constant_medium}

We first compare three realizations for computing the smallest positive eigenvalues of the quasiperiodic Maxwell problem: the original shift-and-invert realization (OSI) for the original GEVP, the nested reduced inverse realization (NRI) for the null-space free problem, and the explicit reduced inverse realization (ERI) based on \eqref{eq:Kr_inverse_orthogonalized}. We set $T_\ell=2\pi$, $\ell=1,\ldots,6$,
\begin{equation}
P=
\begin{pmatrix}
1 & 0 & 0 & \sqrt{2} & 0 & 0\\
0 & 1 & 0 & 0 & \sqrt{3} & 0\\
0 & 0 & 1 & 0 & 0 & \sqrt{5}
\end{pmatrix}^{\mathsf T},
\label{eq:exp1_projection_matrix}
\end{equation}
and use the smooth quasiperiodic permittivity $\mathcal E(\x)=4\exp\!\left(\frac{1.2}{6}\sum_{\ell=1}^{6}\cos x_\ell\right)I_3$. For $N = 5$, $N_F=(2N)^6=10^6$, so that the original GEVP has dimension $3\times10^6$, whereas the null-space free problem \eqref{eq:null_space_free_EVP} has dimension $2\times10^6$.

For the fixed-size comparison in Figure~\ref{fig:exp1_outer_iterations}, we use the closed physical Bloch path
$$
\q^{(0)}=(0.05,0.04,0.03)^{\mathsf T},
\quad
\q^{(1)}=(0.25,0.04,0.03)^{\mathsf T},
\quad
\q^{(2)}=(0.25,0.20,0.03)^{\mathsf T},
$$
$$
\q^{(3)}=(0.25,0.20,0.18)^{\mathsf T},
\qquad
\q^{(4)}=\q^{(0)}.
$$
Each segment is divided into three equal subintervals, giving $12$ independent Bloch-wave-vector problems, and the smallest ten positive eigenvalues are computed at every point. To examine the dependence on the discretization size, we fix $\q=\q^{(0)}$ in Figures~\ref{fig:exp1_linear_iterations} and~\ref{fig:exp1_walltime} and vary $N=3,\ldots,7$.

OSI applies the original shift-and-invert formulation with shift $\sigma=10^{-4}$ through the Arnoldi-based MATLAB \texttt{eigs}, with the shifted systems solved by preconditioned MINRES to relative tolerance $10^{-10}$. NRI and ERI apply inverse Lanczos to the null-space free problem, using the nested realization of $K_r^{-1}$ and the explicit reduced inverse \eqref{eq:Kr_inverse_orthogonalized}, respectively. The CG tolerances are $10^{-10}$ and $10^{-12}$ for the outer and inner systems in NRI, and $10^{-10}$ for the $\widehat M$ system in ERI. The reduced inverse-Lanczos computations use Krylov dimension $m_{\max}=40$ and tolerance $\tau_{\mathrm{Lan}}=10^{-10}$. 

To examine the dependence of the three realizations on the Fourier truncation size, we terminate the computation for a given method if the accumulated number of iterative linear-solver steps reaches $10^6$ or if the elapsed solver time reaches $10^4$ s, and no larger value of $N$ is considered thereafter for that method. The reported wall time includes the method-dependent setup and eigensolution.

\begin{figure}[htbp]
\centering
\captionsetup{skip=3pt}
\captionsetup[subfigure]{skip=2pt}
\begin{subfigure}[t]{0.315\textwidth}
\centering
\includegraphics[width=\textwidth]
{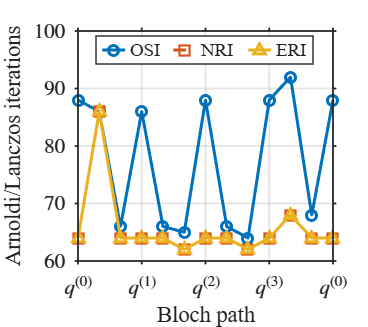}
\caption{Arnoldi/Lanczos iterations on the Bloch path.}
\label{fig:exp1_outer_iterations}
\end{subfigure}
\hfill
\begin{subfigure}[t]{0.315\textwidth}
\centering
\includegraphics[width=\textwidth]
{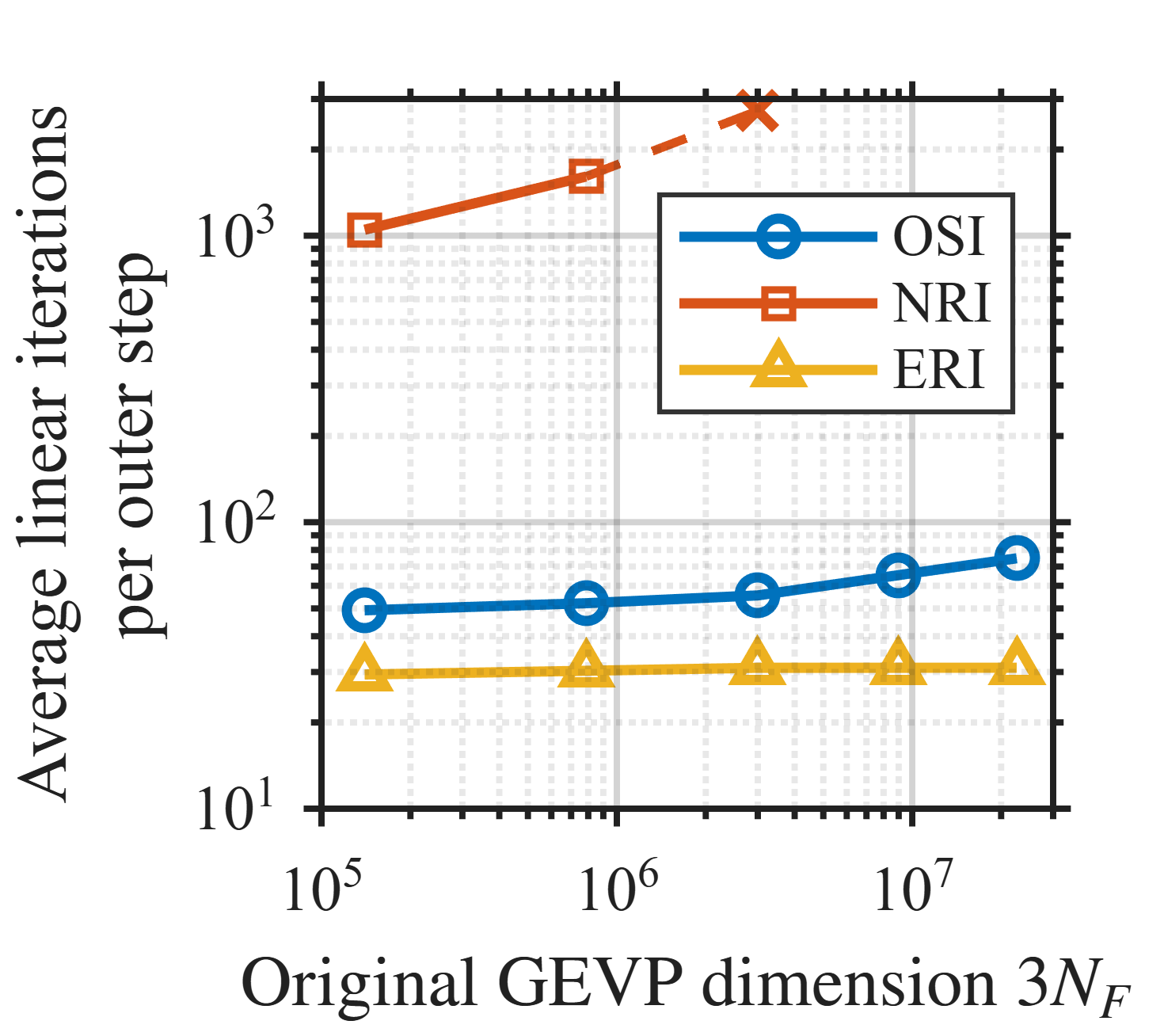}
\caption{Average linear iterations per outer step at $\q=\q^{(0)}$.}
\label{fig:exp1_linear_iterations}
\end{subfigure}
\hfill
\begin{subfigure}[t]{0.315\textwidth}
\centering
\includegraphics[width=\textwidth]
{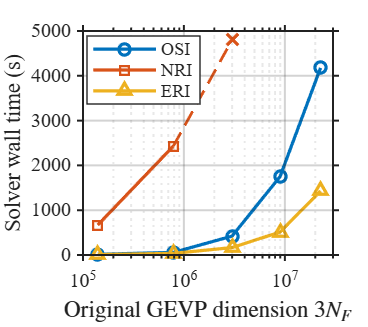}
\caption{Solver wall time versus $3N_F$ at $\q=\q^{(0)}$.}
\label{fig:exp1_walltime}
\end{subfigure}
\caption{Comparison of OSI, NRI, and ERI, with panel (a) evaluated for $N=5$. Panel (a) compares the Arnoldi and Lanczos iterations, while panels (b) and (c) show the average linear iterations per outer step and the solver wall time, respectively, as functions of the original GEVP dimension $3N_F$. The dashed NRI continuation and cross indicate the first incomplete computation; the vertical position of the cross is only illustrative.}
\label{fig:exp1_inverse_comparison}
\end{figure}

The three realizations recover the same positive eigenvalues to numerical precision. Figure~\ref{fig:exp1_outer_iterations} shows that NRI and ERI require exactly the same number of outer Lanczos steps at every Bloch wave vector, with mean counts $65.83$ for both, compared with $76.92$ Arnoldi steps for OSI. This agreement reflects the fact that NRI and ERI solve the same null-space free eigenproblem and differ only in the realization of $K_r^{-1}$. Hence the explicit representation \eqref{eq:Kr_inverse_orthogonalized} leaves the outer spectral iteration unchanged.

The distinction between the two reduced realizations appears in the inner linear solves. As shown in Figure~\ref{fig:exp1_linear_iterations}, the average number of linear iterations per outer step for ERI varies only mildly with the problem dimension, consistent with the conditioning estimate for the orthogonalized inner system, whereas NRI incurs substantially greater inner work because of its nested solves. At $N=5$, NRI reaches the prescribed $10^4$s time limit before completion and is therefore not continued to larger dimensions. Figure~\ref{fig:exp1_walltime} further shows that ERI remains less expensive than OSI as the dimension increases. Thus, the null-space free reduction determines the reduced eigenproblem, whereas the explicit inverse \eqref{eq:Kr_inverse_orthogonalized} removes the nested linear solves and provides the computational advantage observed here.

\subsection{Experiment 2: 3D Yee-field reconstruction}
\label{subsec:exp_Yee_reconstruction}

Using the periods, projection matrix, and permittivity of Subsection~\ref{subsec:exp_constant_medium}, we examine the accuracy, cost, and scalability of \eqref{eq:Taylor_Yee_field_reconstruction}. We compute the first ten positive eigenmodes of the $N = 5$ discretization at $\q=(0.13,0.21,0.17)^{\mathsf T}$ and use the tenth positive mode as the representative field, with $N_F=10^6$. For the accuracy-cost study, we set $h=0.0125$ and $n=40$, so that $\Omega_n=\{-40,\ldots,40\}^3$ and the enlarged Yee grid contains $N_g=(2n+3)^3=571787$ sampling indices for each staggered field component.

The reference field $\widehat{\u}_{\mathrm{dir}}$ is computed by blockwise direct evaluation of \eqref{eq:direct_Yee_field_reconstruction} with batch size $256$. For the multi-center Taylor reconstruction, we consider $p=6,8,10$ and $S=m^3$ local blocks with $m=1,\ldots,6$. Each coordinate index set is divided into $m$ contiguous, nearly equal-sized groups; each staggered component uses the midpoint of its coordinate ranges in each block as the expansion center. Increasing $m$ reduces the local phase radii but requires forming moments at more centers. For a Taylor-reconstructed field $\widehat{\u}_p$, we measure
$$
e_2(p)=
\frac{\|\widehat{\u}_p-\widehat{\u}_{\mathrm{dir}}\|_2}
{\|\widehat{\u}_{\mathrm{dir}}\|_2},
\qquad
e_A(p)=
\frac{\|\widehat A\widehat{\u}_p-\widehat A\widehat{\u}_{\mathrm{dir}}\|_2}
{\|\widehat A\widehat{\u}_{\mathrm{dir}}\|_2}.
$$
The configuration $p=10$ with $64=4^3$ local centers is adopted. For an $m^3$ partition, the maximum phase radius is $r_{\mathrm{ph},\max}=(\ell_{\max}/2)\max_{\boldsymbol\xi\in\mathcal J_v}\|\q_{\boldsymbol\xi}\|_1$, where $\ell_{\max}=h(\lceil(2n+3)/m\rceil-1)$ is the largest local block width. For $n=40$ and $m=4$, we have $\ell_{\max}=0.25$ and $\max_{\boldsymbol\xi\in\mathcal J_v}\|\q_{\boldsymbol\xi}\|_1\approx41.40$, giving $r_{\mathrm{ph},\max}\approx5.18$.

We next vary the window half-width index $n=40,60,\ldots,200$ at fixed $N=5$ and $h=0.0125$. Figure~\ref{fig:taylor_ng_scaling} uses $n^3$ on the horizontal axis; the enlarged grid contains $N_g=(2n+3)^3$ points per field component. The Taylor method uses $p=10$ and $S=m^3$ centers, with the smallest integer $m$ satisfying $r_{\mathrm{ph},\max}\leq5.18$. Its timings include moment formation at all centers and polynomial evaluation. Each method has a $10^4$~s time budget; the cross marks the first direct reconstruction exceeding it. Relative field errors are reported only when both methods complete.

\begin{figure}[t]
\centering
\captionsetup{skip=3pt}
\captionsetup[subfigure]{skip=2pt}

\begin{subfigure}[t]{0.48\textwidth}
\vspace{0pt}
\centering
\includegraphics[
    width=\linewidth,
    trim={0 4mm 0 0},
    clip
]
{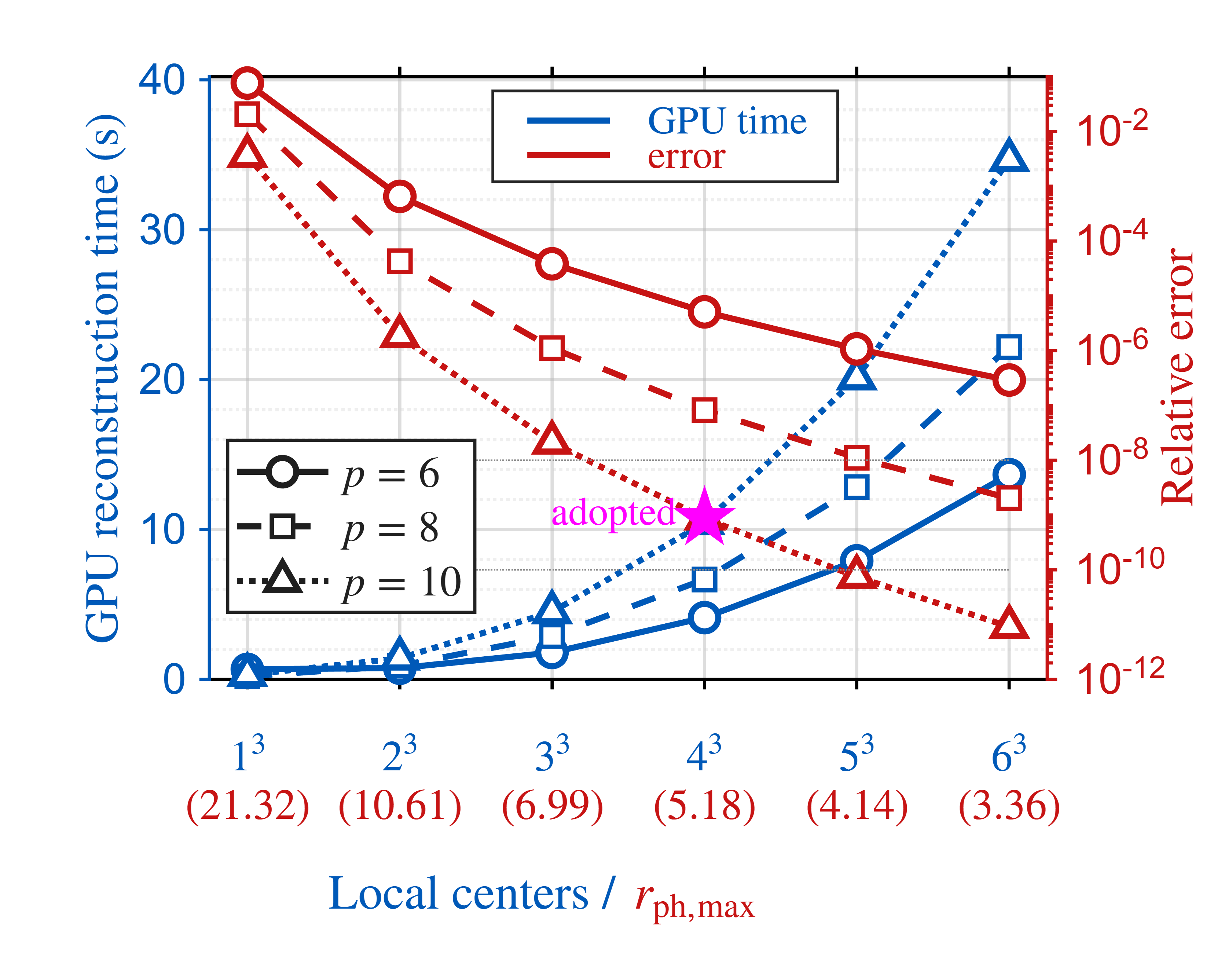}
\caption{Accuracy-cost trade-off of the Taylor reconstruction for $p=6,8,10$ under progressively refined local partitions.}
\label{fig:taylor_order_center_tradeoff}
\end{subfigure}
\hfill
\begin{subfigure}[t]{0.48\textwidth}
\vspace{0pt}
\centering
\includegraphics[
    width=\linewidth,
    trim={0 4mm 0 0},
    clip
]
{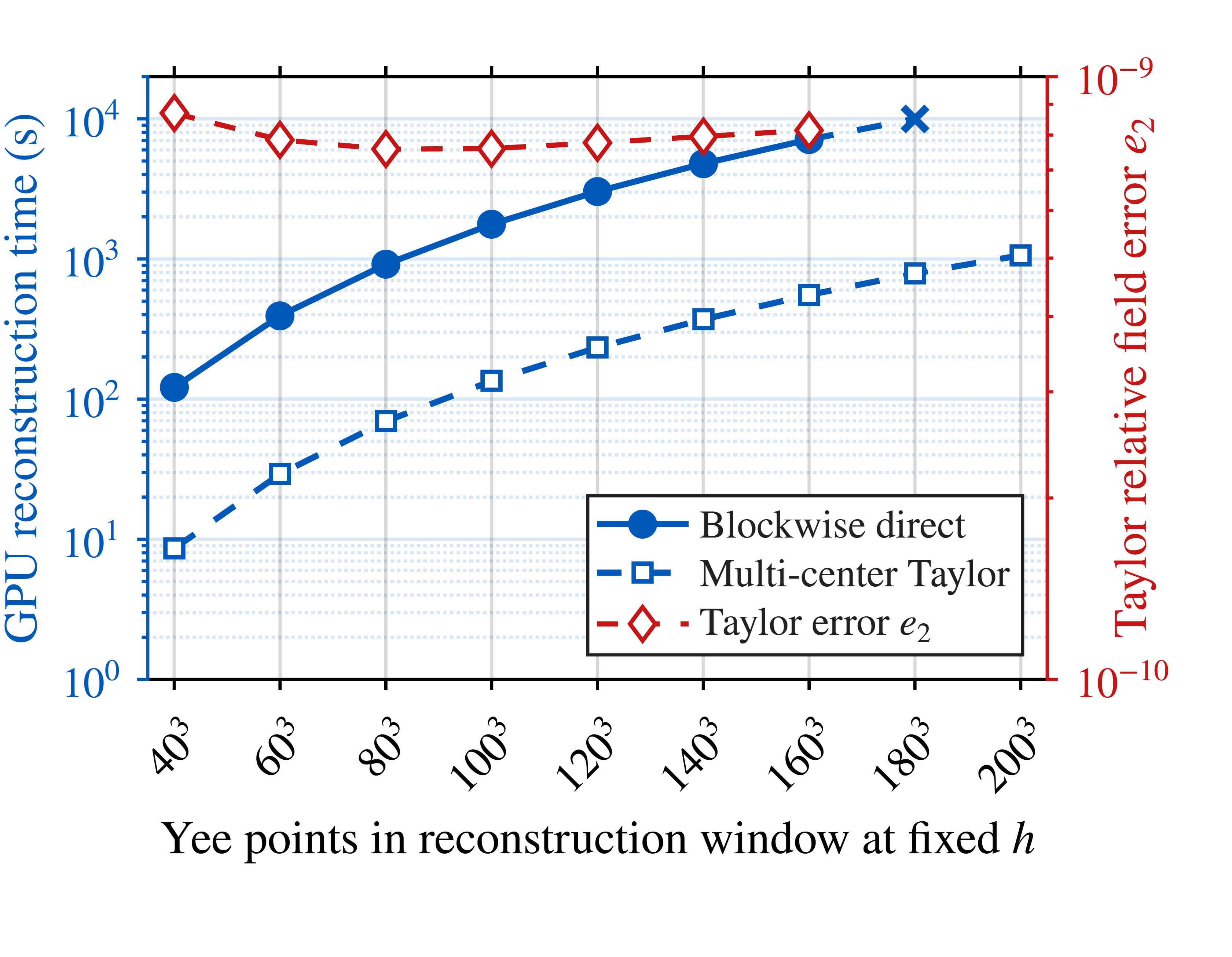}
\caption{Reconstruction scalability of both methods with increasing window size at fixed Yee spacing.}
\label{fig:taylor_ng_scaling}
\end{subfigure}

\caption{Accuracy and scalability of the multi-center Taylor reconstruction. Panel (a) examines the effects of the Taylor order and local partition on reconstruction accuracy and cost, while panel (b) compares the reconstruction cost of the blockwise direct and multi-center Taylor formulations as the physical sampling window, and hence the number of Yee points, increases at fixed Yee spacing.}
\label{fig:taylor_reconstruction_performance}
\end{figure}

Figure~\ref{fig:taylor_order_center_tradeoff} confirms the role of the local phase radius in the truncated reconstruction. For each Taylor order, localization reduces $r_{\mathrm{ph},\max}$ and the corresponding field error, while increasing $p$ improves the accuracy at a fixed partition. The adopted $p=10$, $64$-center setting gives $e_2=8.69\times10^{-10}$ and $e_A=1.38\times10^{-5}$. Thus, the phase-radius criterion motivated by \eqref{eq:Taylor_phase_radius}-\eqref{eq:Taylor_Yee_reconstruction_error} provides an effective parameter for controlling the truncation error of the separated representation.

Figure~\ref{fig:taylor_ng_scaling} shows that the field error remains at a similar level as the sampling window grows under the phase-radius criterion. At the largest size completed by both methods, the Taylor reconstruction is approximately $12.9$ times faster; the direct reconstruction exceeds the $10^4$~s budget at the next size. These results demonstrate the computational advantage of the tested Taylor implementation at fixed $N$ and $h$, with the number of centers increased as required.

\subsection{Experiment 3: Physical-space consistency and LWRQ validation}
\label{subsec:exp_LWRQ_validation}
We use the same medium and physical Bloch vector as in Subsection~\ref{subsec:exp_Yee_reconstruction}. Since the ordered mode index need not remain fixed as the Fourier truncation is enlarged, we use Fourier-overlap tracking to select two persistent modes across the truncations. The first and tenth positive modes at $N=2$ are associated with modes $11$ and $114$, respectively, at $N=5$, with Fourier overlaps $0.9908$ and $0.9484$. We therefore use modes $11$ and $114$ throughout this experiment. For $N=5$, we have $N_F=10^6$. The multi-center Taylor reconstruction uses $p=10$, with the local centers chosen so that $r_{\mathrm{ph},\max}\leq5.2$.

The Gaussian weights use $\sigma=2h$ and $N(\boldsymbol\zeta)
=\{\boldsymbol\eta\in\Omega_n:
\|\boldsymbol\eta-\boldsymbol\zeta\|_\infty\leq3\}$, with the normalization in \eqref{eq:partition_unity_weight}. For each reconstructed Yee field $\widehat{\u}$, we first compute its cropped Yee quotient,
$$
\lambda_{\mathrm{3D}}
=
R_{\Omega_n}(\widehat{\u})
=
\widehat{\lambda}_{\mathrm{LWRQ}},
$$
where the equality follows from Theorem~\ref{thm:weighted_local_RQ_to_global_RQ}. For $\lambda_{\mathrm{3D}}\neq0$, the relative consistency indicators used below are defined by
$$
\eta_{\mathrm{3D}} =
\frac{
\|\widehat A\widehat{\u}
-\lambda_{\mathrm{3D}}\widehat B\widehat{\u}_{\mathrm{o}}\|_2
}{
\|\widehat A\widehat{\u}\|_2
+
|\lambda_{\mathrm{3D}}|
\|\widehat B\widehat{\u}_{\mathrm{o}}\|_2
},
\quad
\nu_{\mathrm{LWRQ}}
=
\frac{
\left(
\sum_{\boldsymbol{\zeta}}
p_{\boldsymbol{\zeta}}
|\rho_{\boldsymbol{\zeta}}-\lambda_{\mathrm{3D}}|^2
\right)^{1/2}
}{
|\lambda_{\mathrm{3D}}|
},\quad
\delta_{\lambda}
=
\frac{ |\lambda_{\mathrm{3D}}-\widetilde{\lambda}| }
{ |\widetilde{\lambda}| }.
$$
Here $\eta_{\mathrm{3D}}$ measures the residual of the reconstructed field in the independently discretized 3D Yee system, while $\nu_{\mathrm{LWRQ}}$ quantifies the spread of the local Rayleigh quotients about the recovered value $\lambda_{\mathrm{3D}}$. The quantity $\delta_{\lambda}$ measures the relative discrepancy between this 3D spectral value and the 6D Fourier eigenvalue $\widetilde{\lambda}$.

We first fix the physical window $[-0.5,0.5]^3$ and use
$$
(h,n)
=
(0.1,5),\
(0.05,10),\
(0.025,20),\
(0.0125,40).
$$
For each mesh, Figure~\ref{fig:LWRQ_physical_sensitivity} reports the maximum values of $\eta_{\mathrm{3D}}$, $\nu_{\mathrm{LWRQ}}$, and $\delta_{\lambda}$ over modes $11$ and $114$. Window sensitivity is examined separately at $h=0.025$ over physical half-widths ranging from $L=0.25$ to $L=4$. For each tracked mode, the field is reconstructed once on the largest window $[-4,4]^3$, and the smaller windows are obtained by centered cropping. The two insets display the representative range $L=1,2,3,4$ separately for modes $11$ and $114$.

\begin{figure}[htbp]
\centering

\begin{subfigure}[t]{0.48\textwidth}
\centering
\includegraphics[width=\linewidth]
{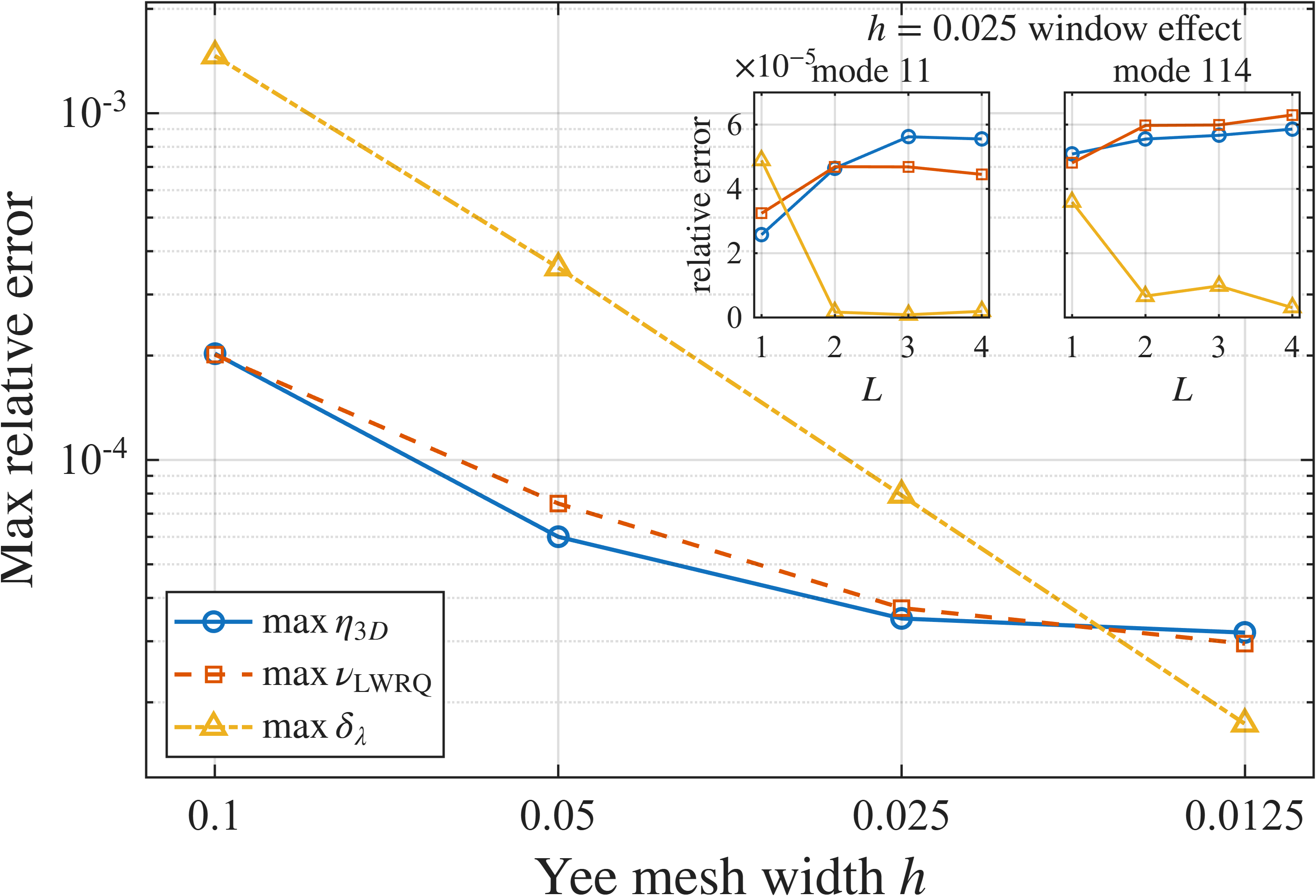}
\caption{Maximum relative consistency indicators under Yee-grid refinement, with window effects shown in the insets.}
\label{fig:LWRQ_physical_sensitivity}
\end{subfigure}
\hfill
\begin{subfigure}[t]{0.48\textwidth}
\centering
\includegraphics[width=\linewidth]
{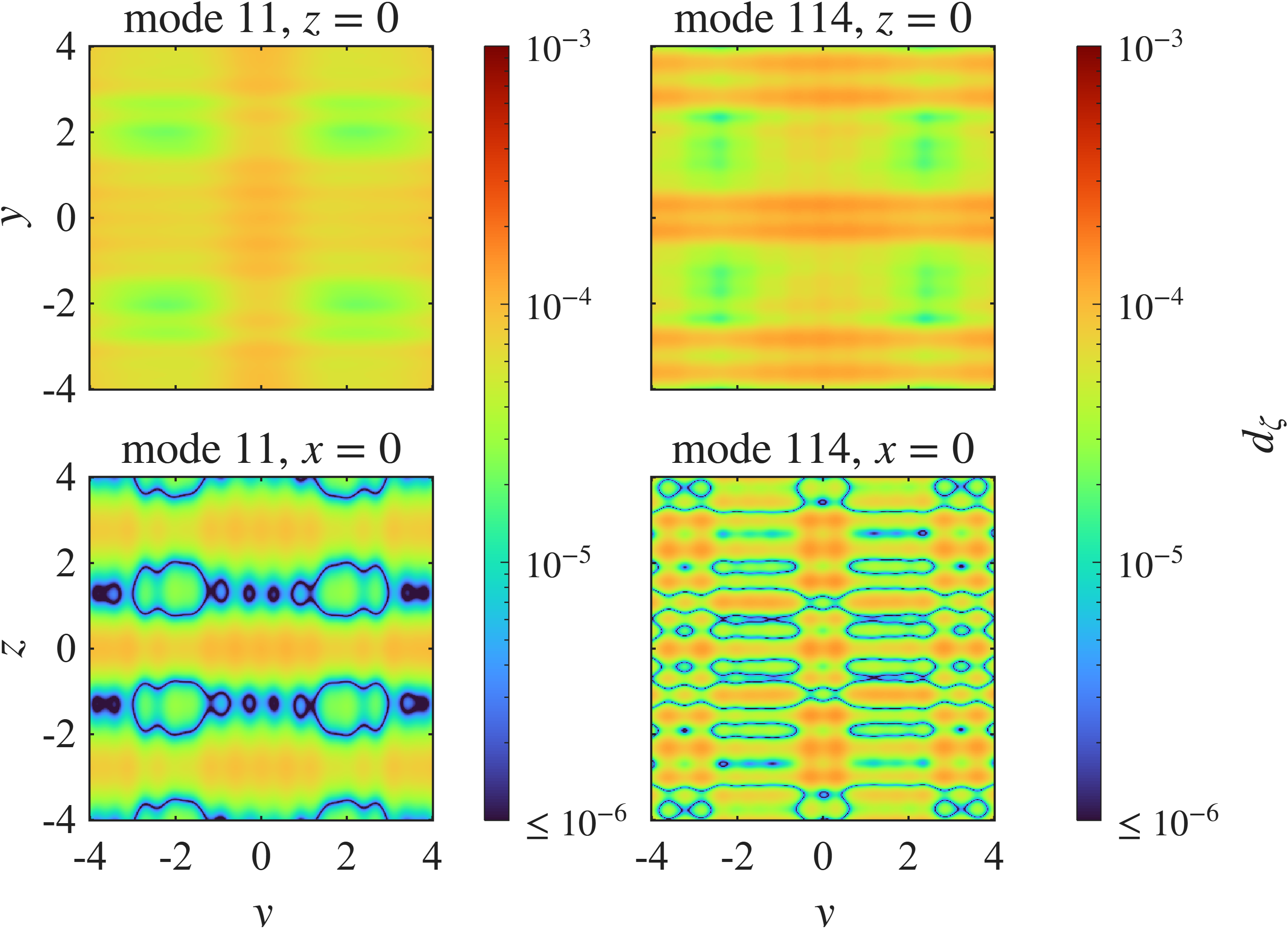}
\caption{Relative local LWRQ deviation for modes $11$ and $114$ on the planes $z=0$ and $x=0$.}
\label{fig:LWRQ_tracked_modes}
\end{subfigure}

\caption{Physical-space consistency of the tracked $N=5$ eigenpairs. Panel~(a) compares the relative Yee residual, LWRQ variation, and 3D-6D spectral discrepancy, while panel~(b) shows the local LWRQ deviations for the two tracked modes on representative physical sections.}
\label{fig:LWRQ_physical_validation}
\end{figure}

Figure~\ref{fig:LWRQ_physical_sensitivity} shows that all three error measures decrease under Yee-grid refinement, indicating improved agreement between the reconstructed 3D fields and the 6D eigenpairs. The two insets further show that the results become insensitive to the physical-window size once the half-width is sufficiently large: for both tracked modes, $\eta_{\mathrm{3D}}$ and $\nu_{\mathrm{LWRQ}}$ remain in the $10^{-5}$ range, while $\delta_{\lambda}$ stays at approximately the $10^{-6}$-$10^{-5}$ level for $L\gtrsim2$. Thus, the observed physical-space consistency is stable with respect to the cropping window.

Together, the small $\eta_{\mathrm{3D}}$ and $\nu_{\mathrm{LWRQ}}$ show that the reconstructed fields behave as approximate eigenfunctions of the independently discretized 3D system, while the decreasing $\delta_{\lambda}$ confirms that the spectral values recovered from these fields agree closely with those computed from the 6D Fourier discretization.

To examine this consistency locally, we further define
$$
d_{\boldsymbol{\zeta}}
=
\frac{
|\rho_{\boldsymbol{\zeta}}-\lambda_{\mathrm{3D}}|
}{
|\lambda_{\mathrm{3D}}|
}.
$$
Figure~\ref{fig:LWRQ_tracked_modes} shows $d_{\boldsymbol{\zeta}}$ for modes $11$ and $114$ on the sections $z=0$ and $x=0$ over $[-4,4]^3$. For each mode, the two sections use a common logarithmic color scale, with values below $10^{-6}$ saturated at the lower limit. The resulting spatial distributions resolve the local variation of the LWRQ deviation and are consistent with the small global $\nu_{\mathrm{LWRQ}}$.

\subsection{Experiment 4: A representative 3D quasiperiodic Maxwell problem}
\label{subsec:exp_real_quasiperiodic}

The final experiment applies the complete computational framework to a representative 3D quasiperiodic Maxwell problem. We compute the low-frequency spectrum along a closed physical Bloch path, reconstruct selected eigenmodes in physical space, and assess their physical-space consistency using cropped Yee residuals and LWRQs.

We use the periods and projection matrix specified in Subsection~\ref{subsec:exp_constant_medium}. On the 6D torus, the lifted permittivity is
\begin{align}
\label{eq:exp4_medium}
\mathcal E(\x)
=
&\,1.5
\Big[
2.8
+
\cos(x_1-x_4)\cos(x_2-x_5)
+
0.8\cos(x_2-x_5)\cos(x_3-x_6)
\notag\\
&\hspace{1.2cm}
+
0.6\cos(x_3-x_6)\cos(x_1-x_4)
+
0.1\sum_{\ell=1}^{6}\cos x_\ell
\Big]I_3 .
\end{align}
Under the physical restriction $\x=P\r$, the added axial harmonics introduce the frequency pairs $(1,\sqrt{2})$, $(1,\sqrt{3})$, and $(1,\sqrt{5})$ in the three coordinate directions. Their irrational ratios exclude any nonzero translation period, while the difference-phase products couple the directional modulations.

The closed physical Bloch path is
$$
\q^{(0)}
=
(0.05,0.04,0.03)^{\mathsf T},
\qquad
\q^{(1)}
=
(0.25,0.04,0.03)^{\mathsf T},
$$
$$
\q^{(2)}
=
(0.25,0.20,0.03)^{\mathsf T},
\qquad
\q^{(3)}
=
(0.25,0.20,0.18)^{\mathsf T},
\qquad
\q^{(4)}
=
\q^{(0)}.
$$
Each segment is divided into $20$ equal subintervals. With $N=5$, we compute the smallest $20$ positive eigenvalues at each Bloch wave vector using the inverse Lanczos method of Section~\ref{sec:inverse_lanczos} and the explicit inverse \eqref{eq:Kr_inverse_orthogonalized}.

To examine the physical-space realization at different path locations and spectral levels, we select
\begin{equation}
\label{eq:exp4_selected_pairs}
S_1=(11,1),
\qquad
S_2=(31,9),
\qquad
S_3=(51,12),
\qquad
S_4=(71,18),
\end{equation}
where the first entry denotes the sampled Bloch-point index and the second denotes the ordered positive eigenmode. For each selected eigenpair, the 3D staggered electric field is reconstructed using the memory-efficient procedure in Subsection~\ref{subsec:memory_efficient_Yee_reconstruction}. The Taylor reconstruction uses $p=10$, with local centers selected to satisfy $r_{\mathrm{ph},\max}\leq5.20$.

For the numerical assessment of physical-space consistency, we use the default Yee parameters from Subsection~\ref{subsec:exp_LWRQ_validation},
$$
h=0.025,
\qquad
n=20,
\qquad
\Omega_n=\{-20,\ldots,20\}^3.
$$
We use the Gaussian neighborhoods and normalized weights specified in Subsection~\ref{subsec:exp_LWRQ_validation}. By Theorem~\ref{thm:weighted_local_RQ_to_global_RQ}, the mass-weighted LWRQ mean equals the cropped Yee quotient, which we compare with the corresponding 6D eigenvalue.

To illustrate the quasiperiodic spatial structure of the reconstructed fields, we additionally plot the normalized intensity $|\widehat{\u}|^2$ on enlarged $z=0$ and $x=y$ slices over $[-2\pi,2\pi]^2$. Each Fourier mode of the 6D eigenvector is mapped to a physical wave vector of the form $\q_{\boldsymbol\xi}=P^{\mathsf T}(\k+\boldsymbol\xi)$, so that the irrational entries of $P$ generate incommensurate spatial frequencies in the reconstructed field. In our previous work on projection-based quasiperiodic Helmholtz eigenproblems~\cite{SunLiLinLyu2026}, we showed that this irrational projection mechanism transforms high-dimensional periodic eigenfunctions into quasiperiodic fields in physical space. The enlarged slices are used only for visualization, while the LWRQ values are computed on the default Yee window above.

\begin{figure}[htbp]
\centering
\captionsetup{skip=3pt}
\captionsetup[subfigure]{skip=2pt}
\begin{subfigure}[t]{0.315\textwidth}
\centering
\includegraphics[width=\textwidth]
{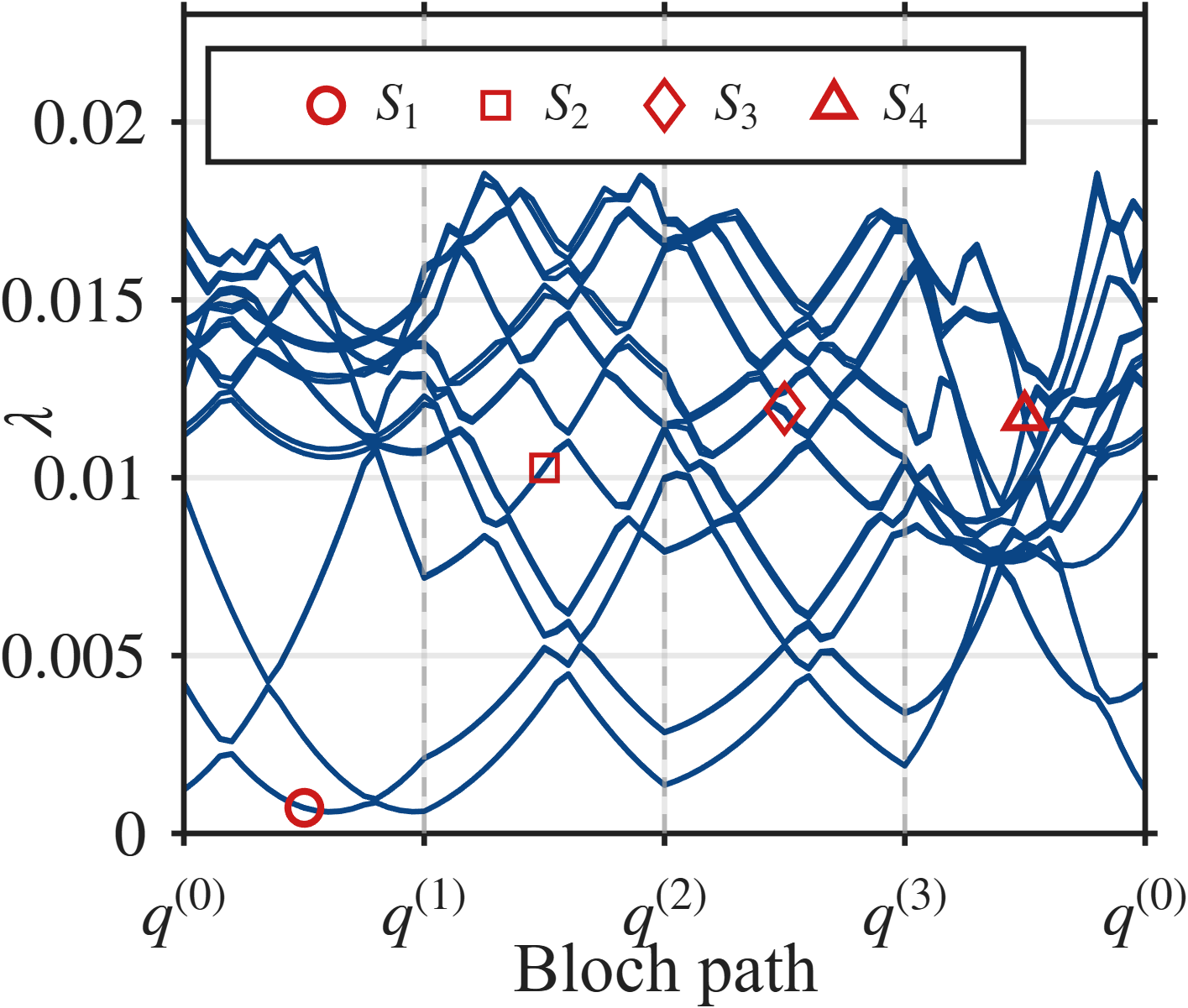}
\caption{Low-frequency spectrum along the closed Bloch path, with LWRQ-recovered values marked at $S_1$-$S_4$.}
\label{fig:exp4_spectrum}
\end{subfigure}
\hfill
\begin{subfigure}[t]{0.315\textwidth}
\centering
\includegraphics[width=\textwidth]
{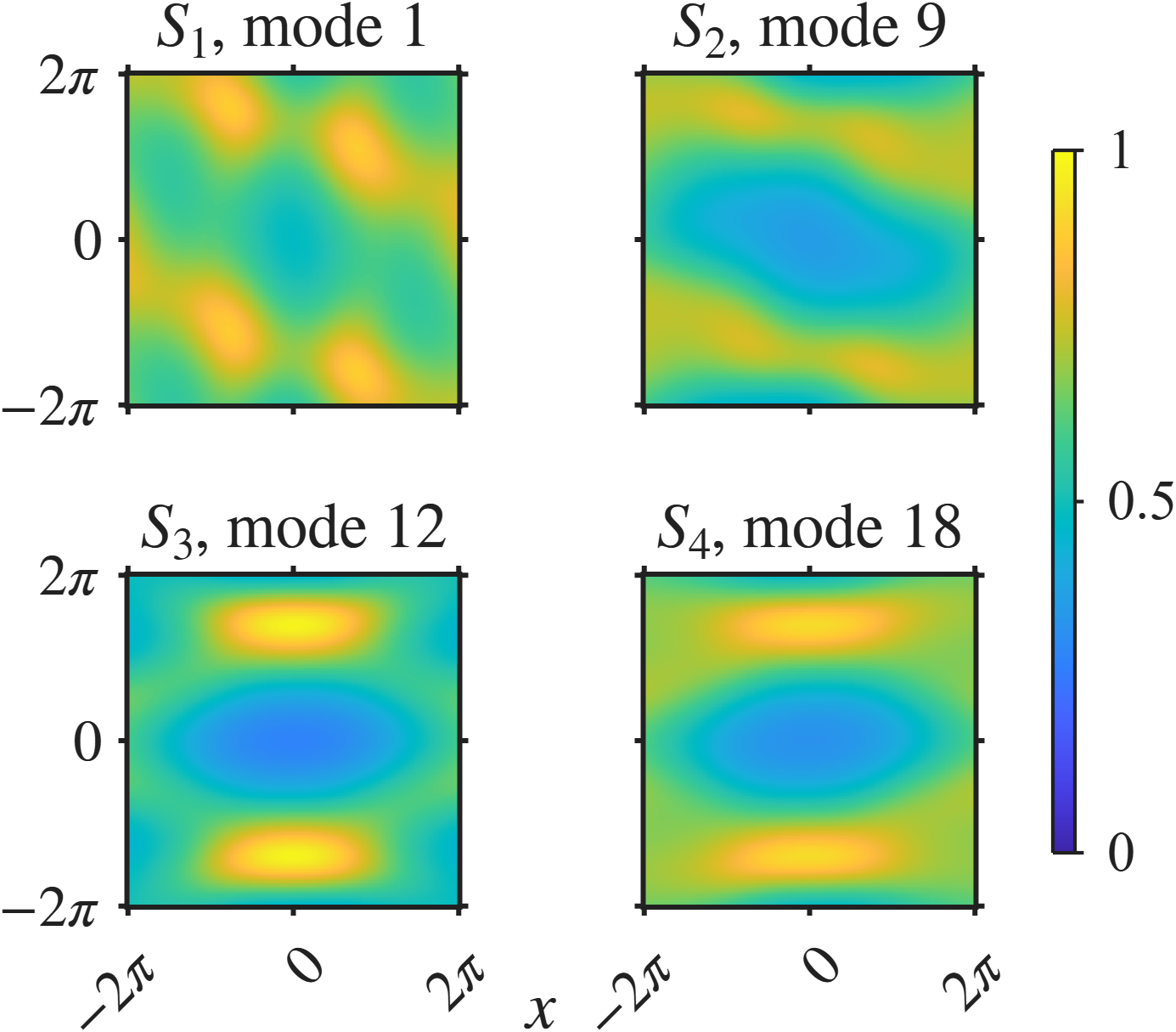}
\caption{Normalized reconstructed electric-field intensities for $S_1$-$S_4$ on $z=0$.}
\label{fig:exp4_field_z0}
\end{subfigure}
\hfill
\begin{subfigure}[t]{0.315\textwidth}
\centering
\includegraphics[width=\textwidth]
{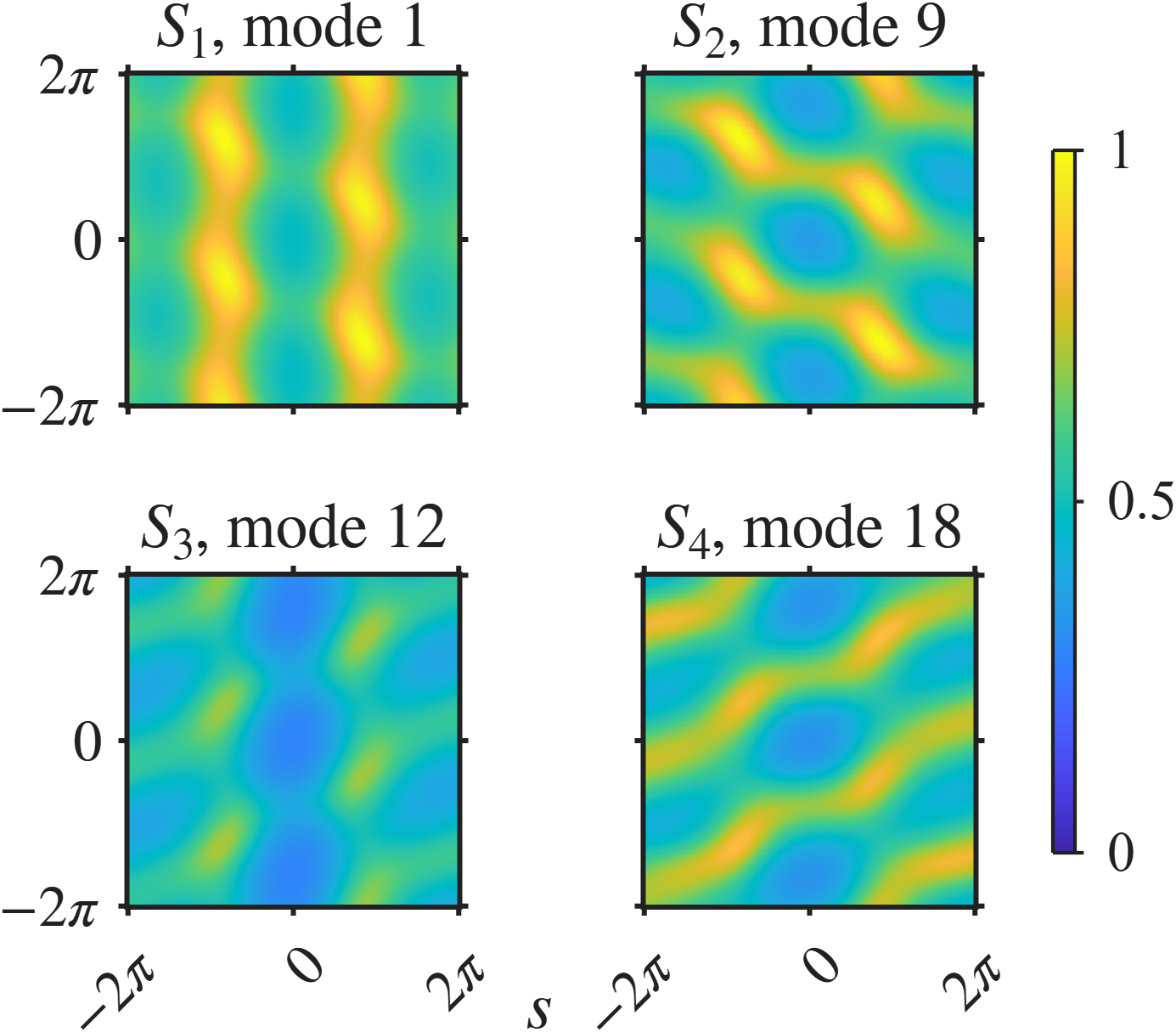}
\caption{Normalized reconstructed electric-field intensities for $S_1$-$S_4$ on $x=y$.}
\label{fig:exp4_field_xy}
\end{subfigure}
\caption{ Spectral and physical-space results for the representative $N=5$ quasiperiodic Maxwell problem. Panel (a) compares the 6D Fourier eigenvalues with the LWRQ-recovered spectral values at the four representative point-mode pairs, whereas panels (b) and (c) show the corresponding normalized reconstructed electric-field intensities on the planes $z=0$ and $x=y$, respectively. }
\label{fig:exp4_quasiperiodic}
\end{figure}

Figure~\ref{fig:exp4_spectrum} shows a densely intertwined computed low-frequency spectrum with apparent crossings, local clustering, and changes in the separation of neighboring branches. Related spectral complexity, including hierarchical and fractal-like structures, has been reported for quasiperiodic photonic systems~\cite{VardenyNahataAgrawal2013,BandresRechtsmanSegev2016,RodriguezEtAl2008}. A related complex spectral organization was also observed in our previous projection-based Helmholtz study~\cite{SunLiLinLyu2026}, where the quasiperiodic spectral branches were recovered from reconstructed physical-space eigenfunctions through weighted pointwise Rayleigh quotients. Here, the red hollow markers represent the mass-weighted LWRQ means obtained from the reconstructed 3D Yee fields, with a maximum relative discrepancy of $1.369\times10^{-3}$ from the 6D eigenvalues over $S_1$-$S_4$.

Figures~\ref{fig:exp4_field_z0} and \ref{fig:exp4_field_xy} display the reconstructed electric-field intensities on the $z=0$ and $x=y$ sections, respectively. The oblique and stripe-like modulations reflect the incommensurate physical frequencies generated by the irrational entries $\sqrt{2}$, $\sqrt{3}$, and $\sqrt{5}$ of the projection matrix through $\q_{\boldsymbol\xi}=P^{\mathsf T}(\k+\boldsymbol\xi)$. The distinct patterns on the two sections further show the directional dependence of the same 3D quasiperiodic eigenmodes.

Together, these results connect the computed low-frequency spectrum, the projected quasiperiodic field structure, and the physical-space consistency assessment through the LWRQ within the same $N=5$ Maxwell computation.

\section{Conclusions}
\label{sec:conclusion}

In this paper, we have developed a computational framework for 3D quasiperiodic Maxwell's eigenvalue problems based on a 6D projection formulation. A modewise longitudinal-transverse decomposition removes the gradient kernel of the projected Bloch-Fourier discretization and reduces the original GEVP to a null-space free standard EVP containing only the positive spectrum. An explicit inverse representation together with an orthogonalized inner formulation enables an efficient inverse Lanczos implementation with a Hermitian positive definite CG system. Beyond the 6D spectral computation, the Fourier eigenvectors are reconstructed on a 3D staggered Yee grid through a separated multi-center Taylor representation, while cropped Yee residuals and LWRQs provide physical-space consistency measures based on a separately discretized Yee operator. The numerical experiments confirm the effectiveness of the proposed reduction, eigensolver, reconstruction, and physical-space recovery framework for 3D quasiperiodic Maxwell problems. Future work will investigate how the projection geometry influences the spectral structure, including comparisons of band structures generated from the same higher-dimensional periodic parent system by different projection matrices, spectral transitions among periodic, quasiperiodic, and more general aperiodic regimes, and nearest-neighbor eigenvalue-spacing statistics for fixed projection matrices.

\bibliographystyle{plain}
\bibliography{research_paper}
\end{document}